\documentclass{amsart}
\usepackage{amsmath,amssymb, amsbsy}
\usepackage[dvips]{graphicx}
\usepackage{color}
\usepackage[latin1]{inputenc}
\usepackage[active]{srcltx}
\usepackage{wrapfig}
\usepackage[dvips]{graphicx}
\usepackage[usenames,dvipsnames]{pstricks}
\usepackage{epsfig}
 \usepackage{pst-grad} 
 \usepackage{pst-plot} 
\usepackage{epic}
\usepackage{epic}
\usepackage{curves}

\usepackage{mathrsfs}

\usepackage{latexsym}
\usepackage{exscale}
\usepackage{pst-all}
\newcommand{\N}{{I\!\!N}}

\newcommand{\dint}{\dyle\int}

\newcommand{\p}{\partial}

\newcommand{\re}{{I\!\!R}}
\newcommand{\ren}{\re^N}

\newcommand{\dyle}{\displaystyle}

\newcommand{\io}{\int\limits_\O}

\renewcommand{\a }{\alpha }

\renewcommand{\d }{\delta }
\newcommand{\D }{\Delta }
\newcommand{\e }{\varepsilon }

\newcommand{\g }{\gamma}

\renewcommand{\l }{\lambda }

\newcommand{\m }{\mu }
\newcommand{\n }{\nabla }

\newcommand{\s }{\sigma }

\renewcommand{\O }{\Omega }

\newcommand{\inn}{\mbox{ in }}

\renewcommand{\ge }{\geqslant}
\renewcommand{\geq }{\geqslant}
\renewcommand{\le }{\leqslant}
\renewcommand{\leq }{\leqslant}

\newenvironment{pf}{\noindent{\sc Proof}.\enspace}{\hfill\qed\medskip}
\newtheorem{Theorem}{Theorem}[section]
\newtheorem{Corollary}[Theorem]{Corollary}
\newtheorem{Lemma}[Theorem]{Lemma}
\newtheorem{Proposition}[Theorem]{Proposition}
\theoremstyle{definition}
\newtheorem{Definition}[Theorem]{Definition}
\newtheorem{remarks}[Theorem]{Remarks}
\newtheorem{remark}[Theorem]{Remark}
\newcommand{\cqd}{{\unskip\nobreak\hfil\penalty50
        \hskip2em\hbox{}\nobreak\hfil\mbox{\rule{1ex}{1ex} \qquad}
        \parfillskip=0pt \finalhyphendemerits=0\par\medskip}}

\begin{document}

\title[Fractional KPZ equations with Hardy potential]
{Fractional KPZ equations with critical growth in the gradient respect to Hardy potential}
\thanks{\thanks{Work partially supported by Project MTM2016-80474-P, MINECO, Spain.\\  The first author is also partially supported by DGRSDT, Algeria. }}
\author[B. Abdellaoui, I. Peral, A. Primo,  F. Soria]{Boumediene Abdellaoui, Ireneo Peral, Ana Primo, Fernando Soria }
\address{\hbox{\parbox{5.7in}{\medskip\noindent {$*$ Laboratoire d'Analyse Nonlin\'eaire et Math\'ematiques
Appliqu\'ees. \hfill \break\indent D\'epartement de
Math\'ematiques, Universit\'e Abou Bakr Belka\"{\i}d, Tlemcen,
\hfill\break\indent Tlemcen 13000, Algeria.}}}}
\address{\hbox{\parbox{5.7in}{\medskip\noindent{Departamento de Matem\'aticas,\\ Universidad Aut\'onoma de Madrid,\\
        28049, Madrid, Spain. \\[3pt]
        \em{E-mail addresses: }{\tt boumediene.abdellaoui@uam.es, \tt ireneo.peral@uam.es, ana.primo@uam.es, fernando.soria@uam.es
         }.}}}}
\date{\today}

 \begin{abstract}

 In this work we study the existence of positive solution to the fractional quasilinear problem,
$$
\left\{
\begin{array}{rcll}
(-\Delta )^s u &=&\lambda \dfrac{u}{|x|^{2s}}+ |\nabla u|^{p}+ \mu f &\inn \Omega,\\
u&>&0 & \inn\Omega,\\
u&=&0 & \inn(\mathbb{R}^N\setminus\Omega),
\end{array}\right.
$$
where $\Omega$ is a $C^{1,1}$ bounded domain in $\mathbb{R}^N$, $N> 2s, \mu>0$, $\frac{1}{2}<s<1$, and $0<\lambda<\Lambda_{N,s}$ is defined in \eqref{bestC}  . We assume  that $f$ is a non-negative function with additional hypotheses.

As we will see, there are deep differences with respect to the case $\lambda=0$. More precisely,
\begin{itemize}
\item If $\lambda>0$, there exists a critical exponent $p_{+}(\lambda, s)$ such that for $p> p_{+}(\lambda,s)$ there is no positive solution.
\item Moreover, $p_{+}(\lambda,s)$ is optimal in the sense that, if $p<p_{+}(\lambda,s)$ there exists a positive solution for suitable data and $\mu$ sufficiently small.
\end{itemize}
 \end{abstract}

\maketitle

\rightline{\textit{To Shair Ahmad in his 85th birthday  with our friendship and recognition. }}

\section{Introduction}
This work deals with the following problem:

\begin{equation}\label{mainn0}
\left\{
\begin{array}{rcll}
(-\Delta )^s u &=&\lambda \dfrac{u}{|x|^{2s}}+ |\nabla u|^{p}+ \mu f &\inn \Omega,\\
u&>&0 & \inn\Omega,\\
u&=&0 & \inn(\mathbb{R}^N\setminus\Omega),
\end{array}\right.
\end{equation}
where $0<\lambda<\Lambda_{N.s}$ defined in \eqref{bestC}, $\mu>0$, $s\in (\frac{1}{2},1)$, $2s<N$, $\Omega\subset \mathbb{R}^N$ is
a bounded regular domain containing the origin and $f$
is a measurable non-negative function satisfying suitable hypotheses.

By $(-\Delta)^s$ we denote the  fractional Laplacian of order $2s$ introduced by M. Riesz in \cite{MRiesz}, that is,
$$(-\Delta)^{s}u(x):=a_{N,s}\mbox{ P.V. }\int_{\mathbb{R}^{N}}{\frac{u(x)-u(y)}{|x-y|^{N+2s}}\, dy},\, \,\qquad  s\in(0,1),$$
where
$$
a_{N,s}=2^{2s-1}\pi^{-\frac N2}\frac{\Gamma(\frac{N+2s}{2})}{|\Gamma(-s)|},
$$
is the normalizing constant that gives the Fourier multiplier identity
$$\mathcal{F}((-\Delta)^s u )(\xi)=|\xi|^{2s}\mathcal{F}(u)(\xi),  \text{for } u\in \mathscr{S}(\mathbb{R}^N).$$
See \cite{FLS} for details. 

For $\lambda=0$, in \cite{APN} (see also  version \cite{APNC}), the authors study natural conditions on $f$ in order to determine the existence  of a positive solution to the problem \eqref{mainn0} depending on the value of $p$. There are three cases: subcritical, $p<2s$, critical $p=2s$ and supercritical $p>2s$.

For $\lambda>0$, the problems studied in this article are related to the following Hardy inequality, proved in \cite{He} (see also \cite{B, FLS, SW, Y} and the monograph \cite{Peral-Soria} for a detailed proof).
\begin{Theorem}\label{DH}{\it (Fractional Hardy inequality).}
For all $u\in \mathcal{C}^{\infty}_{0}(\ren)$ the following inequality holds,
\begin{equation}\label{Hardy}
\dint_{\ren} \,|\xi|^{2s} |\hat{u}|^2\,d\xi\geq
\Lambda_{N,s}\,\dint_{\ren} |x|^{-2s} u^2\,dx,
\end{equation}
where
\begin{equation}\label{bestC}
\Lambda_{N,s}= 2^{2s}\dfrac{\Gamma^2(\frac{N+2s}{4})}{\Gamma^2(\frac{N-2s}{4})}.
\end{equation}
The constant $\Lambda_{N,s}$ is optimal and not attained.
\end{Theorem}
\

Notice that, as it was stated in \cite{FLS}, the fractional Hardy's inequality plays an important role in the proof of the {\it stability of relativistic matter} in a very general setting.

It is clear that the criticality of the inequality is motivated by the homogeneity between the fractional Laplacian and the {\it inverse $2s$-potential.} Moreover, letting $s\to 1$, then one can prove that
$$\Lambda_{N,s}\to \Lambda_{N,1}:=\left(\dfrac{N-2}{2}\right)^2,$$
the classical Hardy constant.

Notice that the optimal constant defined in \eqref{bestC} coincides for every  bounded domain $\Omega$ containing the pole of the Hardy
potential. That is, if $0\in \Omega$, we can rewrite the Hardy inequality \eqref{Hardy} as
\begin{equation}\label{hardy}
\frac{a_{N,s}}{2}\int_{Q}{\frac{|u(x)-u(y)|^2}{|x-y|^{N+2s}}}\, dx\, dy\geq
\Lambda_{N,s}\int_{\Omega}{\frac{u^2}{|x|^{2s}}\,dx},\,u\in H_{0}^{s}(\Omega).
\end{equation}
The optimality of $\Lambda_{N,s}$ here follows by a scaling argument.

Related to problem \eqref{mainn0}, in the local case $s=1$ and for $0<\l<\Lambda_{N,1}$ fixed, the authors in \cite{APS} identify a critical exponent $p_{+}(\lambda)$ such that for $p\geq p_{+}$, there exists no positive weak solution and for $1<p<p_{+}$, $\mu$ sufficiently small, and $f\leq \dfrac{1}{|x|^2}$, they prove the existence of a weak positive solution.

Problem \eqref{mainn0} can be seen as the stationary  Kardar-Parisi-Zhang problem with fractional diffusion and under the influence of the \textit{uncertainty principle} given by the Hardy inequality.  The  classical model by Kardar-Parisi-Zhang was  introduced in \cite{KPZ} with diffusion driven by the Laplacian. In the fractional setting see \cite{W}.

Our aim in this work is to analyze the case $s\in (\frac{1}{2}, 1)$ and $\lambda>0$. Notice that $s>\frac 12$ ensures the ellipticity of the problem. Our main result is the following one.

\begin{Theorem}\label{mainint}
Assume that $s\in (\frac 12,1)$ and $0<\l<\Lambda_{N,s}$, then there exists a critical exponent $p_{+}(\lambda, s)>0$ such that if $p> p_{+}(\lambda,s)$ there is no positive solution to problem \eqref{mainn0}. Moreover, if $p<p_{+}(\lambda,s)$, problem \eqref{mainn0} has a positive solution for suitable data and $\mu$ sufficiently small.
\end{Theorem}

The paper is organized as follows. In Section \ref{2}, we give the notion of solution that we are going to consider here. Moreover, we study the behavior of radial potential solutions of the homogenous problem in the whole space.
Section \ref{3} is devoted to the non existence of solutions. In that respect, we obtain two types of non existence results.
\begin{itemize}
\item On the one hand, we prove the existence of $p_{+}(\lambda, s)$ such that if $s\in (\frac{1}{2}, 1)$ and $p>p_{+}(\lambda, s)$, for all $\lambda>0$, the problem has no positive solution in a weak sense.
\item On the other, we prove that for $s\in (\frac{1}{2}, 1)$, there exists ${\mu^{*}}>0$
such that if $\mu>\mu^{*}$, the problem has no positive solution for any $p$;
that is, the positive source term must be small enough to ensure the existence of solutions.
\end{itemize}
Section \ref{4} is devoted precisely to the existence of solutions.
For  $p<p_{+} (\lambda,s)$ and under additional hypotheses
on the integrability of $f$, we are able to build a suitable supersolution and
then by a monotonicity argument,  to
prove the existence of a minimal positive solution for all $\mu$.
Moreover, for $p<\dfrac{N}{N-2s+1}$, and for all $f\in L^{1}(\Omega)$  that satisfies a suitable integral condition near
the origin, we prove the existence of $\mu^{*}$ such that for $\mu<\mu^{*}$, there exists
a positive solution.

In the last section, we treat the case where the gradient term depends also on a  zero order term. In this case under a suitable behavior of the {\it zero order term} at infinity, we are able to show the existence of a solution for all $p<2s$, under suitable hypotheses on the data. It is worthy to point out that,in the local case, this last problem  comes from the elliptic part of a porous medium equation, see \cite{AGMP}.

\section{Preliminary results}\label{2}
Before starting the analysis of existence and non existence of positive solution,
let us begin describing the precise sense in which solutions are defined.
Consider the problem
\begin{equation}\label{weak0}
\left\{
\begin{array}{rcll}
(-\Delta)^s u &= & g &
\text{ in }\Omega, \\ u&>&0 &\hbox{
in }\Omega,\\ u &=& 0 &\hbox{  in } \mathbb{R}^N\setminus\Omega,
\end{array}%
\right.
\end{equation}
where $g\in L^1(\O)$.
\begin{Definition}\label{weak-sol} We define the class of test functions
\begin{equation}\label{test}
\mathcal{T} (\Omega)=\{ \phi\,\,| \,\, (-\Delta)^s (\phi)=\psi \hbox{ in  } \Omega,\quad \phi=0 \hbox{ in  } \mathbb{R}^N\setminus\Omega, \quad\psi\in \mathcal{C}^\infty_0(\Omega)\}.
\end{equation}
\end{Definition}
Notice that if $v\in \mathcal{T} (\Omega)$ then, using the results in \cite{LPPS},    $v\in H^s_0 (\Omega) \cap L^\infty(\Omega)$. Moreover,  according to the regularity theory developed in \cite{S1}, if $\Omega$ is smooth enough, there exists a  constant $\beta>0$ (that depends only on the structural constants) such that $v\in \mathcal{C}^\beta(\Omega)$ (see also \cite{Kas}).
\begin{Definition}\label{def1}
We say that $u\in L^1(\Omega)$ is a {\it weak solution} to \eqref{weak0} if for  $g\in L^{1}(\Omega)$ we have that
$$\int_\Omega u   \psi dx=\int_\Omega g \phi dx, $$
for any $\phi\in \mathcal{T}  (\Omega)$ with $\psi \in \mathcal{C}^\infty_0(\Omega)$.
\end{Definition}

Recall also the definition of the truncation operator $T_k$,
\begin{equation}\label{Tk}
T_k (\sigma) = \max\{-k ; \min\{ k, \sigma \}\}.
\end{equation}
From \cite{LPPS}, \cite{CV1} and \cite{AAB} we have the next existence result.

\begin{Theorem}\label{entropi}
Suppose that $g\in L^1(\O)$, then problem \eqref{weak0}  has a unique weak solution $u$ obtained as the limit of $\{u_n\}_{n\in \mathbb{N}}$, the sequence of  unique solutions to the approximating problems
\begin{equation}\label{proOO}
\left\{\begin{array}{rcll}
(-\Delta)^s u_n &= & g_n(x) & \mbox{  in  }\O,\\
u_n &= & 0 & \mbox{ in } \ren\backslash\O,
\end{array}
\right.
\end{equation}
with $g_n=T_n(g)$. Moreover,
\begin{equation} \label{tku}
T_k(u_n)\to T_k(u)\hbox{  strongly in }   H^{s}_{0}(\Omega), \quad \forall k > 0,
\end{equation}
\begin{equation} \label{L1u}
u \in L^q \,, \qquad  \forall  \ q\in \big(1, \frac{N}{N-2s}\big)\,
\end{equation}
 and
\begin{equation}\label{L1du}
\big|(-\Delta)^{\frac{s}{2}}   u\big| \in L^ r \,, \qquad \forall  \  r \in \big(1,  \frac{N}{N-s} \big) \,.
\end{equation}
In addition, if $s>\frac 12$, then $u\in W^{1,q}_0(\O)$ for all $q<\frac{N}{N-(2s-1)}$ and $u_n\to u$ strongly in $W^{1,q}_0(\O)$.
\end{Theorem}

Now, before dealing with the main problem \eqref{main0}, let us introduce the following definition.
\begin{Definition}\label{def2}
Assume that $f\in L^1(\O)$ is a nonnegative function. We say that $u$ is
a solution to problem \eqref{mainn0} if $u\in W^{1,p}_0(\Omega), \dfrac{u}{|x|^{2s}}\in L^1(\O)$ and, setting
$g\equiv \lambda \dfrac{u}{|x|^{2s}}+|\n u|^p+f$, then $u$ is
a weak solution to problem \eqref{weak0} in the sense of Definition \ref{def1}.
\end{Definition}

In order to study the behavior in a neighborhood of the origin of a nonnegative solution to problem \eqref{mainn0}, we need to analyze each radial \textit{potential} positive solution in the whole space. More precisely, let us consider the homogeneous problem
\begin{equation}\label{problemHardy}
(-\Delta)^{s} u=\l\dfrac{\,u}{|x|^{2s}} \mbox{ in } \mathbb{R}^N\setminus\{0\},
\end{equation}
where $0<\lambda\leq \Lambda_{N,s}$. Then we have (see for instance \cite[Theorem 4.1]{SW})
\begin{Lemma} \label{singularity} Let $0<\lambda\leq \Lambda_{N,s}$. Then $v_{\pm\alpha_{\lambda}}(x)=|x|^{-\frac{N-2s}{2}\pm\alpha_{\lambda}}$ are
solutions to problem \eqref{problemHardy}, where $\alpha{_\lambda}$ is obtained by the identity
\begin{equation}\label{lambda}
\lambda=\lambda(\alpha_{\lambda})=\lambda(-\alpha_{\lambda})=\dfrac{2^{2s}\,\Gamma(\frac{N+2s+2\alpha_{\lambda}}{4})\Gamma(\frac{N+2s-2\alpha_{\lambda}}{4})}{\Gamma(\frac{N-2s+2\alpha_{\lambda}}{4})\Gamma(\frac{N-2s-2\alpha_{\lambda}}{4})}.
\end{equation}
\end{Lemma}

\begin{remark}
Notice that $\lambda(\alpha)= \lambda(-\alpha)=m_{\alpha_{\lambda}}m_{-\alpha_{\lambda}}$, with $m_{\alpha_{\lambda}}= 2^{\alpha_{\lambda}+s}\dfrac{\Gamma(\frac{N+2s+2\alpha_{\lambda}}{4})}{\Gamma(\frac{N-2s-2\alpha_{\lambda}}{4})}$.
\end{remark}

\begin{Lemma}\label{LambdaVsSing}
The following equivalence holds true:
$$
0<\lambda(\alpha_{\lambda})=\lambda(-\alpha_{\lambda})\leq \Lambda_{N,s}\mbox{ if and only if } 0\leq \alpha_{\lambda}<\dfrac{N-2s}{2}.
$$
\end{Lemma}
For an elementary proof of this Lemma see \cite{AMPP, FLS, He}.

\begin{remark}\label{gamma1}
Denote
\begin{equation}\label{g1}
\mu(\l)= \dfrac{N-2s}{2}-\alpha_{\lambda} \hbox{ and } \bar\mu(\l)= \dfrac{N-2s}{2}+\alpha_{\lambda}.
\end{equation}
For $0<\lambda<\Lambda_{N,s}$, then $0<\mu< \dfrac{N-2s}{2}<\bar\mu<(N-2s)$. Since $N-2\mu-2s={2}\alpha_{\lambda}>0$ and $N-2\bar\mu-2s=-{2}\alpha_{\lambda}<0$, then {$(-\Delta)^{s/2}(|x|^{-\mu})\in L^2(\Omega)$, but $(-\Delta)^{s/2}(|x|^{-\bar\mu})$ does not.}
\end{remark}
As a consequence we have the next comparison lemma.
\begin{Lemma}\label{estim0}
Assume that $u\in L^1_{loc}(\ren)$ is such that $u\ge 0$ in $\ren$ with $(-\D)^su\in L^1_{loc}(\O)$. Suppose that
$$
(-\Delta)^{s} u\ge \l\dfrac{\,u}{|x|^{2s}} \mbox{ in } \Omega,\quad  0<\lambda<\Lambda_{N,s},
$$
then
$$
u(x)\ge C|x|^{-\mu(\lambda)}= C|x|^{-\frac{N-2s}{2}+\alpha_{\lambda}} \mbox{  in } B_r(0)\subset\subset \O.
$$
\end{Lemma}

See \cite{AMPP} for a detailed proof.

\section{Non existence result} \label{3}
We now consider the problem stated in the introduction
\begin{equation}\label{main0}
\left\{
\begin{array}{rcll}
(-\Delta)^s u &= & \l\dfrac{u}{|x|^{2s}}+ |\nabla u|^{p}+ \mu f &
\text{ in }\Omega, \\ u&>&0 &\hbox{
in }\Omega,\\ u &=& 0 &\hbox{  in } \mathbb{R}^N\setminus\Omega,
\end{array}%
\right.
\end{equation}
where $\Omega \subset \ren$ is a bounded regular domain containing the origin,  $0<\lambda<\Lambda_{N,s}$,  $\mu>0, s\in (\frac 12,1)$, $p>1$ and  $f$ is a non-negative function.

To establish the upper bound for $p$ we follow closely the arguments of \cite{APS}, see also \cite{brezdup} for the potential case. We look for a radial solution to the problem \eqref{main0}.

\begin{equation}\label{soulrn}
(-\Delta)^s w-\lambda \frac{w}{|x|^{2s}}=|\nabla w|^p \mbox{  in }\ren.
\end{equation}
In particular, if we choose $w=A|x|^{\beta-\frac{N-2s}{2}}$, with $A$ a positive constant, $0<\beta<\dfrac{N-2s}{2}$, then \eqref{soulrn} is equivalent to have
$$A\gamma_\beta|x|^{-2s-\frac{N-2s}{2}+\beta}-\lambda A|x|^{\beta-2s-\frac{N-2s}{2}}=\dfrac{A\left|\beta-\dfrac{N-2s}{2}\right| ^{p}}{|x|^{(\frac{N-2s}{2}- \beta +1)p}},$$
where
\begin{equation}\label{gamma}
\gamma_\beta:=\gamma_{-\beta}:=\frac{2^{2s}\Gamma(\frac{N+2s+2\beta}{4})\Gamma(\frac{N+2s-2\beta}{4})}{\Gamma(\frac{N-2s-2\beta}{4})\Gamma(\frac{N-2s+2\beta}{4})}.
\end{equation}
Hence, in order to have homogeneity we need $$
p=\dfrac{\frac{N-2s}{2}-\beta+2s}{\frac{N-2s}{2}-\beta+1},
$$
which means that $\beta=\frac{N-2s}{2}+ \frac{p}{p-1}-\frac{2s}{p-1}$
and, then, the constants must satisfy the equation
$\gamma_\beta-\lambda=A^{p-1} \left|\beta-\dfrac{N-2s}{2} \right| ^{p}.$

Since $A>0$, we need $\gamma_\beta-\lambda>0$.
Consider the map
$$\begin{array}{rcl}\Upsilon:(-\frac{N-2s}{2}, \frac{N-2s}{2})&\mapsto& (0,\Lambda_{N,s})\\
\beta&\mapsto& \gamma_{\beta}
\end{array}
$$
then $\Upsilon$ is even and the restriction of $\Upsilon$ to the set $[0, \frac{N-2s}{2})$ is decreasing, see \cite{DaDuMon} and \cite{FLS},  so there exists a unique $\alpha_\lambda\in (0, \Lambda_{N,s}]$ such that  $\gamma_{\alpha_{\lambda}}=\gamma_{-\alpha_{\lambda}}=\lambda$.

Let $\beta_0=-\beta_1=\a_\l$, therefore, setting
$$
p_+(\lambda, s):=\dfrac{\frac{N-2s}{2}-\beta_0+2s}{\frac{N-2s}{2}-\beta_0+1}=
\frac{N+2s-2\alpha_{\lambda}}{N-2s-2\alpha_{\lambda}+2},
$$
and
$$
p_-(\lambda, s):=\dfrac{\frac{N-2s}{2}-\beta_1+2s}{\frac{N-2s}{2}-\beta_1+1}
=\frac{N+2s+2\alpha_\l}{N-2s+2\alpha_\l+2},
$$
it holds that
 $p_-(\lambda, s)<p_+(\lambda, s)$ and $\gamma_\beta-\lambda>0$ if and only if
$$
p_-(\lambda, s)<p<p_+(\lambda, s).
$$
It is easy to check that $p_{+}(\lambda,s)$ and $p_-(\lambda,s)$ are respectively an increasing and a decreasing function in $\alpha_{\lambda}$ and, therefore, are respectively a decreasing and an increasing function in $\lambda$. Thus
$$
\dfrac{N}{N-2s+1}< p_{-}(\lambda,s)< \dfrac {N+2s}{N-2s+2}< p_{+}(\lambda,s)<2s, \mbox{ for } 0<\lambda<\Lambda_{N,s}.
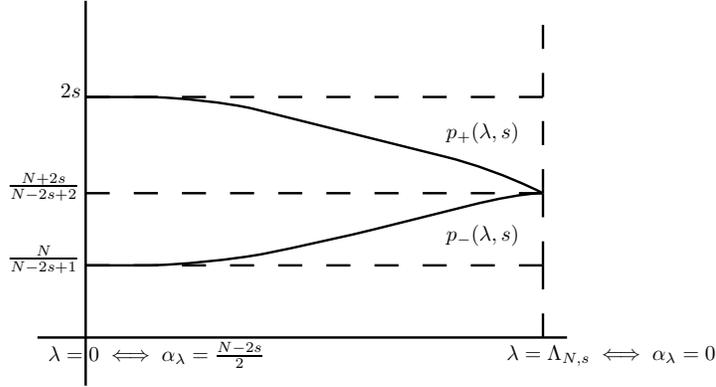
$$
\vskip 0.5cm
\begin{center}
\scalebox{0.8} 
{
\begin{pspicture}(0,-3.22)(14.70664,3.22)
\psline[linewidth=0.04cm](2.3994727,3.2)(2.3994727,-3.2)
\psline[linewidth=0.04cm](1.5994726,-2.4)(10.399472,-2.4)
\psline[linewidth=0.04cm](2.3994727,1.6)(2.7994726,1.6)
\psline[linewidth=0.04cm](3.1994727,1.6)(3.5994728,1.6)
\psline[linewidth=0.04cm](3.9994726,1.6)(4.3994727,1.6)
\psline[linewidth=0.04cm](4.799473,1.6)(5.1994724,1.6)
\psline[linewidth=0.04cm](5.5994725,1.6)(5.9994726,1.6)
\psline[linewidth=0.04cm](6.3994727,1.6)(6.799473,1.6)
\psline[linewidth=0.04cm](7.1994724,1.6)(7.5994725,1.6)
\psline[linewidth=0.04cm](7.9994726,1.6)(8.399472,1.6)
\psline[linewidth=0.04cm](8.799473,1.6)(9.199472,1.6)
\psline[linewidth=0.04cm](9.599473,1.6)(9.999473,1.6)
\psline[linewidth=0.04cm](2.3994727,-1.2)(2.7994726,-1.2)
\psline[linewidth=0.04cm](3.1994727,-1.2)(3.5994728,-1.2)
\psline[linewidth=0.04cm](3.9994726,-1.2)(4.3994727,-1.2)
\psline[linewidth=0.04cm](4.799473,-1.2)(5.1994724,-1.2)
\psline[linewidth=0.04cm](5.5994725,-1.2)(5.9994726,-1.2)
\psline[linewidth=0.04cm](6.3994727,-1.2)(6.799473,-1.2)
\psline[linewidth=0.04cm](7.1994724,-1.2)(7.5994725,-1.2)
\psline[linewidth=0.04cm](7.9994726,-1.2)(8.399472,-1.2)
\psline[linewidth=0.04cm](8.799473,-1.2)(9.199472,-1.2)
\psline[linewidth=0.04cm](9.599473,-1.2)(9.999473,-1.2)
\psline[linewidth=0.04cm](2.3994727,0.0)(2.7994726,0.0)
\psline[linewidth=0.04cm](3.1994727,0.0)(3.5994728,0.0)
\psline[linewidth=0.04cm](3.9994726,0.0)(4.3994727,0.0)
\psline[linewidth=0.04cm](4.799473,0.0)(5.1994724,0.0)
\psline[linewidth=0.04cm](5.5994725,0.0)(5.9994726,0.0)
\psline[linewidth=0.04cm](6.3994727,0.0)(6.799473,0.0)
\psline[linewidth=0.04cm](7.1994724,0.0)(7.5994725,0.0)
\psline[linewidth=0.04cm](7.9994726,0.0)(8.399472,0.0)
\psline[linewidth=0.04cm](8.799473,0.0)(9.199472,0.0)
\psline[linewidth=0.04cm](9.599473,0.0)(9.999473,0.0)
\usefont{T1}{ptm}{m}{n}
\rput(2.149326,1.705){$2s$}
\usefont{T1}{ptm}{m}{n}
\rput(1.7,0.105){$\frac{N+2s}{N-2s+2}$}
\usefont{T1}{ptm}{m}{n}
\rput(1.7,-1.095){$\frac{N}{N-2s+1}$}
\pscustom[linewidth=0.04]
{
\newpath
\moveto(2.3994727,1.6)
\lineto(3.3994727,1.6)
\curveto(3.8994727,1.6)(4.799473,1.5)(5.1994724,1.4)
\curveto(5.5994725,1.3)(6.3994727,1.1)(6.799473,1.0)
\curveto(7.1994724,0.9)(7.9994726,0.7)(8.399472,0.6)
\curveto(8.799473,0.5)(9.399472,0.3)(9.999473,0.0)
}
\pscustom[linewidth=0.04]
{
\newpath
\moveto(2.3994727,-1.2)
\lineto(3.3994727,-1.2)
\curveto(3.8994727,-1.2)(4.8994727,-1.1)(5.3994727,-1.0)
\curveto(5.8994727,-0.9)(6.799473,-0.7)(7.1994724,-0.6)
\curveto(7.5994725,-0.5)(8.399472,-0.3)(8.799473,-0.2)
\curveto(9.199472,-0.1)(9.699472,0.0)(9.999473,0.0)
}
\psline[linewidth=0.04cm](9.999473,2.8)(9.999473,2.4)
\psline[linewidth=0.04cm](9.999473,2.0)(9.999473,1.6)
\psline[linewidth=0.04cm](9.999473,1.2)(9.999473,0.8)
\psline[linewidth=0.04cm](9.999473,0.4)(9.999473,0.0)
\psline[linewidth=0.04cm](9.999473,-1.2)(9.999473,-1.6)
\psline[linewidth=0.04cm](9.999473,-2.0)(9.999473,-2.4)
\psline[linewidth=0.04cm](9.999473,-1.2)(9.999473,-0.8)
\psline[linewidth=0.04cm](9.999473,0.0)(9.999473,-0.4)
\usefont{T1}{ptm}{m}{n}
\rput(11.132793,-2.695){$\lambda=\Lambda_{N,s}\iff\alpha_{\lambda}=0$}
\usefont{T1}{ptm}{m}{n}
\rput(3.5738087,-2.695){$\lambda=0\iff\alpha_{\lambda}=\frac{N-2s}{2}$}
\usefont{T1}{ptm}{m}{n}
\rput(9,1){$p_{+}(\lambda,s)$}
\usefont{T1}{ptm}{m}{n}
\rput(9,-0.695){$p_{-}(\lambda,s)$}
\end{pspicture}}
\end{center}
\vskip 0.5cm
Recalling that
$\mu(\l)= \dfrac{N-2s}{2}-\alpha_{\l}, \bar{\mu}(\lambda)= \dfrac{N-2s}{2}+\alpha_{\lambda}$, then

$$p_+(\lambda, s)=\dfrac{\mu(\l)+2s}{\mu(\l)+1} \hbox{ and  }
p_-(\lambda, s)=\dfrac{\bar{\mu}(\l)+2s}{\bar{\mu}(\l)+1}.$$
Therefore, if $p_-(\lambda, s)<p<p_+(\lambda, s)$ we will be able to construct a radial supersolution for the Dirichlet  problem \eqref{main0} under suitable condition on $f$, just modifying the $w$ found above. Hence this bound for $p$ will be the threshold for the existence also for the Dirichlet problem.
\begin{remark}
Notice that for $s=1$,
$$\lambda_{\alpha_\lambda}:=\frac{2^{2}\Gamma(\frac{N+2+2\alpha_{\lambda}}{4})\Gamma(\frac{N+2-2\alpha_{\lambda}}{4})}{\Gamma(\frac{N-2-2\alpha_{\lambda}}{4})\Gamma(\frac{N-2+2\alpha_{\lambda}}{4})}= 4\Big(\dfrac{N-2+2\alpha_{\lambda}}{4}\Big)\Big(\dfrac{N-2-2\alpha_{\lambda}}{4}\Big) = 4 \Big(\dfrac{N-2}{4}\Big)^2-\alpha_{\lambda}^2.$$
Hence $\alpha_{\lambda}= \pm \sqrt {\Big(\dfrac{N-2}{2}\Big)^2-\lambda }$ and $p< \dfrac{\frac{N+2}{2}-\alpha_{\lambda}}{ \frac{N}{2}-\alpha_{\lambda}}= \dfrac{2+\frac{N-2}{2}-\alpha_{\lambda}}{1+\frac{N-2}{2}-\alpha_{\lambda}}=\dfrac{2+\alpha_1}{1+\alpha_{1}}=p_{+}(\lambda).$
This coincides with the nonexistence exponent defined in \cite{APS}.
\end{remark}
The first part of the main non existence result in Theorem \ref{mainint}, related to the size of the exponent of the nonlinear term is the following.
\begin{Theorem}\label{non1}
Assume that  $s\in (\frac 12, 1)$ and $p>p_+(\l,s)$. For $\l>0$, problem \eqref{main0} has no positive solution $u$ in the sense of Definition \ref{def2}.
\end{Theorem}
\begin{proof}
We argue by contradiction. Assume that $u$ is a positive solution to \eqref{main0} in the sense of Definition \ref{def2}, then $u\in W^{1,p}_0(\O)$ and $\dfrac{u}{|x|^{2s}}\in L^1(\O)$. By Lemma \ref{estim0}, it follows that
$$
u(x)\ge C|x|^{-\mu(\l)}\mbox{  in   }B_r(0)\subset\subset \O.
$$
The proof will be given in several steps according to the value of $p$.

\

{\bf First case: $p\ge \dfrac{N}{\mu(\l)+1}>p_+(\l,s)$}.

Since $\dfrac{u}{|x|^{2s}}, |\n u|^p\in L^1(B_r(0))$, then  by the
Poincaré-Wirtinger inequality we conclude that $u\in W^{1,p}(L^1(B_r(0))$.
Thus $u\in L^{p^*}(B_r(0))$, $p^*=\dfrac{Np}{N-p}$.
Therefore, due to the behavior of $u$ near the origin, we conclude that
$|x|^{-\mu(\l)}\in L^{p^*}(B_r(0))$. Thus,
$p^* \mu(\l)< N$, namely, $p< \frac{N}{\mu(\l)+1}$, which is a contradiction with the condition on $p$.

\

{\bf Second case: $2\le p<\dfrac{N}{\mu(\l)+1}$}.

Since $u\in W^{1,p}_0(\O)\subset W^{1,2}_0(\O)$, then $u\in H^s_0(\O)$.
It is well known that if $u$ is a solution in the sense of Definition \ref{def2},
then $u$ is an entropy solution \eqref{main0}. Hence we can use $T_k(u)$, the truncation function of $u$, as a test function in \eqref{main0} to conclude that
$$
\begin{array}{lll}
\dyle \io|\n u|^pT_k(u) dx+\l\io \frac{uT_k(u)}{|x|^{2s}}dx &\le & \dyle \iint_{\mathbb{R}^{N}\times \mathbb{R}^{N}}\frac{(u(x)-u(y))(T_k(u(x))-T_k(u(y)))}{|x-y|^{N+2s}}\, dy\, dx\\
&\le & ||u||^2_{H^s_0(\O)}\le C(\O)||u||^2_{W^{1,2}_0(\O)}.
\end{array}
$$
Letting $k\to \infty$ and using Fatou's lemma, we reach that
$$
\io u|\n u|^p dx+\l\io \frac{u^2}{|x|^{2s}}dx\le ||u||^2_{H^s_0(\O)}\le C(\O)||u||^2_{W^{1,2}_0(\O)}.
$$
By using the H\"{o}lder and Young inequalities we find that
$$
\io|\n u|^2u dx\le C_1 \io|\n u|^pu dx+C_2\io u<\infty,
$$
and then we conclude that $u^{\frac{3}{2}}\in W^{1,2}_0(\O)$ and then $u^{\frac{3}{2}}\in H^s_0(\O)$. As above, using $u^2$ as a test function in \eqref{main0} and using the fact that
$$
\bigg(u^2(x)-u^2(y)\bigg)\bigg(u(x)-u(y)\bigg)\le C(u^{\frac{3}{2}}(x)-u^{\frac{3}{2}}(y))^2,
$$
it follows that
$$
\io|\n u|^pu^2dx+\l\io \frac{u^3}{|x|^{2s}}dx\le C ||u^{\frac{3}{2}}||^2_{H^s_0(\O)}<\infty.
$$
Iterating the above process, it holds that
$$
\io|\n u|^pu^mdx+\l\io \frac{u^{m+1}}{|x|^{2s}}dx<\infty,\mbox{  for all  }m.
$$
Choosing $(m+1)\mu(\l)+2s\ge N$, we reach a contradiction and then the non existence result follows in this case too.

\

{\bf Third case: $2s<p<2$}.
We follow  the same idea as in the second case.
Since $u\in W^{1,p}_0(\O)$, then $u\in W^{\s,p}_0(\O)$ for all $\s<1$.
Setting $\s=\frac{2s}{p}$, then $u\in  W^{\frac{2s}{p},p}_0(\O)$.

We claim that \textit{if $u$ is a solution to \eqref{main0}, then
$\dyle \io|\n u|^pu^{a} dx<\infty$, for all $a>0$.}

To prove the claim we begin by noticing that, since $(p-1)<1$, then for all $m>0$, we have the next algebraic inequality,
\begin{equation}\label{alg}
(a-b)(a^{pm-1}-b^{pm-1})\le C(p,m)|a^m-b^m|^{p}, \mbox{ for all  }a,b\ge 0.
\end{equation}
Thus using an approximation argument and by taking $u^{p-1}$ as a test function in \eqref{main0}, using the algebraic inequality \eqref{alg} with $m=1$, it holds that
\begin{eqnarray*}
\dyle \io|\n u|^pu^{p-1}dx+\l\io \frac{u^p}{|x|^{2s}}dx &\le & \iint_{D_\O}\frac{(u(x)-u(y))\bigg(u^{p-1}(x)-u^{p-1}(y)\bigg)}{|x-y|^{N+\frac{2s}{p}p}}dxdy\\
&\le &\dyle C(p)\iint_{D_\O}\frac{|u(x)-u(y)|^p}{|x-y|^{N+\frac{2s}{p}p}}dxdy=C(p)||u||^p_{W^{\frac{2s}{p},p}_0(\O)}<\infty.
\end{eqnarray*}
Thus $\dyle \io|\n u|^pu^{p-1}dx<\infty$, and then $u^{\frac{2p-1}{p}}\in W^{1,p}_0(\O)$ and, as a consequence, $u^{\frac{2p-1}{p}}\in W^{\frac{2s}{p},p}_0(\O)$.
Observe here that $\frac{2p-1}{p}>1$. Now we set $m_1=\frac{2p-1}{p}$;
then choosing $u^{pm_1-1}$ as a test function in \eqref{main0}  (again using an approximation argument), it follows that
\begin{eqnarray*}
\dyle \io|\n u|^pu^{pm_1-1} dx+\l\io \frac{u^{pm_1+1}}{|x|^{2s}}dx &\le & \iint_{D_\O}\frac{(u(x)-u(y))\bigg(u^{pm_1-1}(x)-u^{pm_1-1}(y)\bigg)}{|x-y|^{N+\frac{2s}{p}p}}dxdy\\
\le \!\!\!\dyle\quad C(p)\iint_{D_\O}\frac{|u^{m_1}(x)-u^{m_1}(y)|^p}{|x-y|^{N+\frac{2s}{p}p}}dxdy&=&C(p)||u^{m_1}||^p_{W^{\frac{2s}{p},p}_0(\O)}\!\!<\infty.
\end{eqnarray*}
Thus $\dyle \io|\n u|^pu^{pm_1-1} dx<\infty$ and then
$u^{\frac{3p-2}{p}}\in W^{1,p}_0(\O)$. As a consequence,
$u^{\frac{3p-2}{p}}\in W^{\frac{2s}{p},p}_0(\O)$.
Setting $m_{j+1}=m_j+(1-\frac 1p)$ and iterating the above process,
it holds that $$\dyle \io|\n u|^pu^{pm_j-1} dx<\infty,$$
for all $j$. Since $m_j\to \infty$ as $j\to \infty$, then we conclude that $\dyle \io|\n u|^pu^{a}dx<\infty$ for all $a>0$ and the claim follows.

Therefore, we obtain that $u^{\frac{a}{p}+1}\in W^{1,p}_0(\O)$, for all $a>0$.
Thus using the local Hardy inequality in the space $W^{1,p}_0(\O)$ we reach that
$$
C(N,p)\io \frac{u^{p(\frac{a}{p}+1)}}{|x|^p}dx\le ||u^{\frac{a}{p}+1}||^p_{W^{1,p}_0(\O)}<\infty.
$$
Choosing $a\ge \frac{N-p(\mu(\l)+1)}{\mu(\l)}$,  we reach a contradiction.

Thus the non existence result follows again in this case.

\

{\bf Fourth case: $p_+(\l,s)<p\le 2s$}.

We deal now with the range $p_+(\l,s)<p\le 2s$, which is more involved.

Recall that $p_+(\l,s)=\frac{\mu(\l)+2s}{\mu(\l)+1}$.
Since $\l>0$, then $p<2s+\mu(\l)$. We closely follow an argument used in \cite{AA}.

Let us consider the set of functions $\mathrm{T}(\O)$ defined by
\begin{equation}\label{set}
\mathrm{T}(\O):=\{\theta \in \mathcal{C}_0(\Omega)\mbox{  with  }\theta \gneqq 0 \mbox{  and  }\text{Supp}\, \theta\subset B_r(0)\subset\subset \O\}.
\end{equation}
Let $\theta\in \mathrm{T}(\O)$ be fixed and define $\phi_{\theta} \in H_0^s(\Omega) \cap L^\infty(\O)$, the unique solution of the problem
\begin{equation} \label{eqphi0}
\left\{
\begin{aligned}
(-\D)^s \phi_{\theta} & = \theta\,, & \quad \textup{ in } \Omega\,, \\
\phi_{\theta} & = 0\,, & \textup{ in } \ren \setminus {\Omega}.
\end{aligned}
\right.
\end{equation}
Then, according with \cite{ROS}, $\phi_{\theta}\simeq \delta^s$, where $\delta(x)$ denotes the distance to the boundary.

Using $\phi_{\theta}$ as test function in \eqref{main0}, it holds that
\begin{equation}\label{eq1}
\lambda \int_{\Omega} \frac{u \phi_{\theta}\,}{|x|^{2s}}dx+ \int_{\Omega} |\nabla u |^p \phi_{\theta}\, dx< \int_{\Omega} u(-\D)^s\, \phi_{\theta}\, dx = \int_{\Omega} u\, \theta \, dx.
\end{equation}
Now, consider $\psi_{\theta} \in W_{loc}^{1,p}(\Omega)$ to be the unique solution to the problem
\begin{equation} \label{eqpsi}
\left\{
\begin{aligned}
-\text{div}( \phi^a_{\theta} |\nabla \psi_{\theta}|^{p-2} \nabla \psi_{\theta}) & = \theta\,, & \quad \textup{ in } \Omega\,,\\
\psi_{\theta} & = 0\,, & \textup{ on } \partial \Omega\,,
\end{aligned}
\right.
\end{equation}
where $a<\frac{p-1}{s}$. Existence of $\psi_\theta$ will be proved in Lemma \ref{aux} below.

\

Going back to \eqref{eq1} and using Young's inequality, we have that

\begin{eqnarray*}
\dyle \lambda \int_{\Omega} \frac{u \phi_{\theta}\,}{|x|^{2s}}dx+ \int_{\Omega} |\nabla u |^p \phi_{\theta}\, dx &\le & \int_{\Omega} u\, \theta \, dx=  \int_{\Omega} u \left( -\text{div}( \phi^a_{\theta} |\nabla \psi_{\theta}|^{p-2} \nabla \psi_{\theta}) \right) dx \\
& \le & \dyle\int_{\Omega} \phi^a_{\theta} |\nabla u | |\nabla \psi_{\theta}|^{p-1} dx \\
&\le & \dyle \frac 12\int_{\Omega} \phi_{\theta} |\n u|^p dx + C_2 \int_{\Omega} \phi^{(a-1)p'+1}_{\theta} |\nabla \psi_{\theta}|^p dx.
 \end{eqnarray*}
Thus, we get
$$
\lambda \int_{\Omega} \frac{u \phi_{\theta}\,}{|x|^{2s}}dx \leq C_2\int_{\Omega} \phi^{(a-1)p'+1}_{\theta} |\nabla \psi_{\theta}|^p dx,$$
with $C_2$  depending only on $p$. Due to the behavior of $u$ near the origin, we get
\begin{equation}\label{eq00}
\lambda \int_{B_r(0)} \frac{\phi_{\theta}\,}{|x|^{2s+\mu(\l)}}dx \leq C_2\int_{\Omega} \phi^{(a-1)p'+1}_{\theta} |\nabla \psi_{\theta}|^p dx.
\end{equation}
Setting
$$
\mathrm{Q}(\theta):=\dfrac{\dyle \int_{\Omega} \phi^{(a-1)p'+1}_{\theta} |\nabla \psi_{\theta}|^p}{\dyle \int_{B_r(0)}\frac{\phi_{\theta}\,}{|x|^{2s+\mu(\l)}}dx},
$$
then $\mathrm{Q}(\theta)=\mathrm{Q}(\mu\theta)$ for all $\mu>0$. Thus
$$
\l^*=\inf_{\{\theta \in \mathrm{T}(\O)\}}\dfrac{\dyle \int_{\Omega} \phi^{(a-1)p'+1}_{\theta} |\nabla \psi_{\theta}|^p}{\dyle \int_{B_r(0)}\frac{\phi_{\theta}\,}{|x|^{2s+\mu(\l)}}dx} \ge C_3\,\l>0,$$
where $C_3$ depends only on $p$. Notice that, going back to inequality \eqref{eq00} and since $\l>0$ is fixed, then if $\dyle
\int_{\Omega} \phi^{(a-1)p'+1}_{\theta} |\nabla \psi_{\theta}|^p<\infty$, it holds that $\dyle\int_{B_r(0)} \frac{\phi_{\theta}\,}{|x|^{2s+\mu(\l)}}dx<\infty$ which will be the key in order to get the desired contradiction.

Notice that, using a suitable approximation and density argument, inequality \eqref{eq00} holds for all $\theta\in L^1(\O)$ with $\text{Supp}\,(\theta)\subset\subset \O$ if, in addition, we can show that $\dyle
\int_{\Omega} \phi^{(a-1)p'+1}_{\theta} |\nabla \psi_{\theta}|^pdx<\infty$. This will be the main idea in order to get the desired results.

Without loss of generality we can assume that $B_1(0)\subset \subset \O$. Consider
$$
\theta(x)=(\frac{1}{|x|^m}-1)_+ \mbox{  with } \max\{2s, N-\mu(\l)\}<m<N
$$
and define $\phi_\theta$, the unique solution to problem \eqref{eqphi0} (that can be considered in a very weak sense or entropy sense). Then
$$
\phi_\theta\backsimeq \frac{\delta^s(x)}{|x|^{m-2s}}.
$$
Since $m-2s<N-p$ and $a<\dfrac{p-1}{s}<1$,
then the weight $\dfrac{1}{|x|^{a(m-2s)}}$ is admissible in the sense of Caffarelli-Kohn-Nirenberg inequalities.

We claim that the auxiliary problem \eqref{eqpsi} has a solution $\psi_\theta$ such that, under the above condition on $m$ and $p$, we have $\dyle
\int_{\Omega} \phi^{(a-1)p'+1}_{\theta} |\nabla \psi_{\theta}|^p<\infty$. Since $\text{Supp}\,\theta\subset\subset \O$ and  $\theta, \phi_\theta$ are only singular at the origin, then we have to show that $\dyle
\int_{B_r(0)} \phi^{(a-1)p'+1}_{\theta} |\nabla \psi_{\theta}|^p<\infty$.

\

Define $\widetilde{\psi_\theta}$ to be the unique solution to the problem
\begin{equation} \label{eqpsim}
\left\{
\begin{aligned}
-\text{div}( \dfrac{1}{|x|^{a(m-2s)}}|\nabla \widetilde{\psi_\theta}|^{p-2} \nabla \widetilde{\psi_\theta}) & = \dfrac{1}{|x|^{m}}\,, & \quad \textup{ in } \Omega\,,\\
\widetilde{\psi_\theta} & = 0\,, & \textup{ on } \partial \Omega\,.
\end{aligned}
\right.
\end{equation}
By a direct computation we can show that, as $x\to 0$, $\widetilde{\psi_\theta}\backsimeq \dfrac{C(\O)}{|x|^{\frac{(1-a)(m-2s)}{p-1}+\frac{2s-p}{p-1}}}$ and
$$
|\n \widetilde{\psi_\theta}|\backsimeq \dfrac{C(\O)}{|x|^{\frac{(1-a)(m-2s)}{p-1}+\frac{2s-1}{p-1}}}.
$$
Then
$$
(\frac{1}{|x|^{m-2s}})^{(a-1)p'+1}|\n \widetilde{\psi_\theta}|^p \approx \dfrac{C(\O)}{|x|^{(m-2s)+p'(2s-1)}},
$$
in a neighborhood of the origin.
Since $m>N-\mu(\l)$, then $(m-2s)+p'(2s-1)<N$, if and only if $p>p_+(\l,s)$.

Thus $$\io (\frac{1}{|x|^{m-2s}})^{(a-1)p'+1}|\n \widetilde{\psi_\theta}|^p <\infty \hbox{ and  then  }
\int_{B_r(0)} \phi^{(a-1)p'+1}_{\theta} |\nabla { \psi_{\theta}} |^p<\infty.$$
However, notice that, by a direct computation,
$$\dyle\int_{B_r(0)} \frac{\phi_{\theta}\,}{|x|^{2s+\mu(\l)}}dx=\infty,$$
 which is a contradiction.
\end{proof}

To finish with the proof of Theorem \ref{non1} we need to show the existence of a solution for the $p$-Laplacian weighted problem \eqref{eqpsi}, as was stated in the fourth case considered above. This is the content of the following lemma.

\begin{Lemma}\label{aux}
Let $\Omega \subset \ren$, $N \geq 2$, be a bounded domain with boundary $\partial \Omega$ of class $\mathcal{C}^2$. Assume that $\theta\in \mathrm{T}(\O)$ defined in \eqref{set} and let $\phi_{\theta} \in H_0^s(\Omega) \cap L^\infty(\O)$ be the solution of
\begin{equation}
\left\{
\begin{aligned}
(-\D)^s\phi_{\theta} & = \theta\,, & \quad \textup{ in } \Omega\,, \\
\phi_{\theta} & = 0\,, & \quad \textup{ in } \ren \setminus {\Omega}\,.
\end{aligned}
\right.
\end{equation}
Suppose that $1<p\le 2s$ and $a<\dfrac{p-1}{s}$. Then there exists $\psi_{\theta} \in W_{loc}^{1,p}(\Omega)$
distributional solution  of
\begin{equation} \label{psq}
\left\{
\begin{aligned}
-\text{div}( \phi^a_{\theta} |\nabla \psi_{\theta}|^{p-2} \nabla \psi_{\theta}) & = \theta\,, & \quad \textup{ in } \Omega\,,\\
\psi_{\theta} & = 0\,, & \textup{ on } \partial \Omega\,.
\end{aligned}
\right.
\end{equation}
Moreover $\phi_{\theta}\in {W^{1,p}_0(\d^{as}(x) dx, \O)}$ where ${ W^{1,p}_0(\d^{as}(x) dx, \O)}$ is the completion of $\mathcal{C}^\infty_0(\O)$ with respect to the norm
$$
||\phi||_1^p:= \int_{\Omega} \delta^{as}(x) |\nabla v|^p dx\,\mbox{  where    }\d(x):=\mbox{dist}(x,\p\O).
$$
\end{Lemma}

Before proving Lemma \ref{aux}, let us recall the next weighted Hardy inequality proved in \cite{Ne}, Theorem 1.6.
\begin{Theorem}\label{w-hardy}
Assume that $\Omega \subset \ren$, $N \geq 2$, is a bounded regular domain and let $0<\s<p-1$. Then, there exists a positive constant $C=C(\O,p,\s)$ such that for all $v \in \mathcal{C}^\infty_0(\O)$, we have
$$
\int_{\Omega} \delta^{\s-p}(x)|v|^p dx \leq C \int_{\Omega} \delta^\s(x) |\nabla v|^p dx\,.
$$
\end{Theorem}

{\bf Proof of Lemma \ref{aux}.}
By the results of \cite{ROS}, we know that
$$
C_1 \delta^{s}(x) \leq \phi_{\theta}(x) \leq C_2 \delta^{s}(x) \,, \qquad \forall\ x \in \Omega,
$$
with $C_1,C_2>0$. For $n \in \N$, we consider the approximate problems
\begin{equation} \label{Pn}
\left\{
\begin{aligned}
-\text{div}\left( \left(\phi_{\theta} + \frac{1}{n} \right)^a|\nabla \psi_n|^{p-2} \psi_n \right) & = \theta\,, & \quad \textup{ in } \Omega\,,\\
\psi_n & = 0\,, & \textup{ on } \partial \Omega\,.
\end{aligned}
\right.
\end{equation}
It is clear that the existence of $\psi_n$ follows using classical variational argument where we obtain also that $\psi_n \in W_0^{1,p}(\Omega) \cap L^\infty(\O)$. Using $\psi_n$ as test function in \eqref{Pn}, we get
\begin{equation} \label{es1}
C_1 \int_{\Omega} \delta^{as}(x) |\nabla \psi_n|^p dx \leq \int_{\Omega} \theta \psi_n (x) dx.
\end{equation}
Hence by Theorem \ref{w-hardy} and choosing $\s=as\in (p-1,p)$, it holds that
\begin{equation} \label{es11}
C_1 \int_{\Omega} \delta^{as}(x) |\nabla \psi_n|^p dx\le ||\theta||_{L^{p'}(\O)}||w||_{L^{p}(\O)}\le C||\theta||_{L^{p'}(\O)}\bigg(\int_{\Omega} \delta^{as}(x) |\nabla \psi_n|^p dx\bigg)^{\frac{1}{p}}.
\end{equation}
Thus $\int_{\Omega} \delta^{as}(x) |\nabla \psi_n|^p dx\le C\mbox{  for all }n.$ Hence $\{\psi_n\}_n$ is bounded in the space ${ W^{1,p}_0(\d^{as}(x) dx, \O)}$. Therefore,
using again  Theorem \ref{w-hardy}, it holds that $\dyle\int_{\Omega} \frac{|\psi_n|^p}{\d^{p-as}}dx\le C$, for all $n$.

\noindent Then, up to a subsequence, we get the existence of
$\psi \in { W^{1,p}_0(\d^{as}(x) dx, \O)}$ such that $\psi_n \rightharpoonup \psi$ weakly in ${W^{1,p}_0(\d^{as}(x) dx, \O)}$, and then $\psi_n \to \psi$ in $L^{\sigma}_{loc}(\Omega)$ for all $1 \leq \sigma < p^*$ and $\psi_n \to \psi$ a.e. in $\Omega$.

\noindent Using Vitali's lemma we can prove that $\psi_n\to \psi$ strongly in $L^p(\Omega)$.

\noindent It is not difficult to show that $\psi$ is a distributional solution to problem \eqref{psq}.

Let show that $\psi_n \to \psi$ strongly in $W^{1,p}_0(\O, \d^{as}(x)dx)$.

Using $(\psi_n - \psi)$ as test function in \eqref{Pn} and having into account that
$$
\io \theta(\psi_n-\psi) dx\to 0\mbox{  as  }n\to \infty,
$$
it follows that
\begin{equation} \label{es4}
\begin{aligned}
C_1 \io \delta^{as}(x)|\nabla \psi_n|^{p-2} \nabla \psi_n \nabla (\psi_n - \psi) dx \leq o(1).
\end{aligned}
\end{equation}
Since
$$
\io \delta^{as}(x)|\nabla \psi_n|^{p-2} \nabla \psi_n \nabla (\psi_n - \psi) dx =\io \delta^{as}(x)\bigg(|\nabla \psi_n|^{p-2} \nabla \psi_n-|\nabla \psi|^{p-2} \nabla \psi \bigg) \nabla (\psi_n - \psi) dx +o(1),
$$
then the result follows.  In a similar way one can show the uniqueness of the solution in the space  ${ W^{1,p}_0(\d^{as}(x) dx, \O)}$.
\qed

\

To finish this section we prove the next non existence result for $\mu$ large.

\begin{Theorem}\label{non21}
Assume that  $s\in (\frac 12, 1)$, then there exists $\mu^*>0$ such that if $\mu>\mu^*$, the problem \eqref{main0} has no positive solution $u$ in the sense of Definition \ref{def2}.
\end{Theorem}
\begin{proof}
We follow closely the proof of the fourth case in Theorem \ref{non1}, see also \cite{AA}. Assume that $u$ is a positive solution to problem \eqref{main0} in the sense of Definition \ref{def2}. Let $\theta\in \mathcal{C}^\infty_0(\O)$ be a nonnegative fixed function and define $\phi_{\theta} \in H_0^s(\Omega) \cap L^\infty(\O)$, the unique solution of the problem
\begin{equation} \label{eqphi01}
\left\{
\begin{aligned}
(-\D)^s \phi_{\theta} & = \theta\,, & \quad \textup{ in } \Omega\,, \\
\phi_{\theta} & = 0\,, & \textup{ in } \ren \setminus {\Omega}.
\end{aligned}
\right.
\end{equation}
Since $\theta$ is bounded, according with \cite{ROS}, then $\phi_{\theta}\simeq \d^s$.

Using $\phi_{\theta}$ as test function in \eqref{main0}, it holds that
\begin{equation}\label{eq10}
\mu \int_{\Omega} f(x)\phi_\theta(x) dx+ \int_{\Omega} |\nabla u |^p \phi_{\theta}\, dx\le \int_{\Omega} u(-\D)^s\, \phi_{\theta}\, dx = \int_{\Omega} u\, \theta \, dx.
\end{equation}
Now, define $\psi_{\theta} \in W_{loc}^{1,p}(\Omega)$ as the unique solution to the problem
\begin{equation} \label{eqpsi1}
\left\{
\begin{aligned}
-\text{div}( \phi_\theta |\nabla \psi_{\theta}|^{p-2} \nabla \psi_{\theta}) & = \theta\,, & \quad \textup{ in } \Omega\,,\\
\psi_{\theta} & = 0\,, & \textup{ on } \partial \Omega\,.
\end{aligned}
\right.
\end{equation}
Notice that the existence of $\psi_\theta$ follows using the same kind of estimates as in the proof of Lemma \ref{aux}.
Hence
$$
\mu \int_{\Omega} f(x)\phi_\theta(x) dx+ \int_{\Omega} |\nabla u |^p \phi_{\theta}\, dx\le \int_{\Omega} u\, \theta \, dx=-\text{div}( \phi_\theta |\nabla \psi_{\theta}|^{p-2} \nabla \psi_{\theta}).
$$
Thus
\begin{eqnarray*}
\dyle \mu \int_{\Omega} f(x)\phi_\theta(x) dx+ \int_{\Omega} |\nabla u |^p \phi_{\theta}\, dx &\le & \int_{\Omega} \phi_\theta |\nabla \psi_{\theta}|^{p-1}|\nabla u |dx  \phi_{\theta}\, dx\\
&\le & \dyle \e \int_{\Omega} |\nabla u |^p \phi_{\theta}\, dx +C(\e) \int_{\Omega} \phi_\theta |\nabla \psi_{\theta}|^{p}\, dx.
\end{eqnarray*}
Choosing $\e$ small it holds that
$$
\dyle \mu \int_{\Omega} f(x)\phi_\theta(x) dx \le C(\e) \int_{\Omega} \phi_\theta |\nabla \psi_{\theta}|^{p}\, dx.
$$
Thus
$$
\mu\le \inf_{\{\theta \in \mathcal{C}_0^{\infty}(\Omega), \theta\ge 0\}}\dfrac{\dyle\int_{\Omega} \phi_{\theta} |\nabla \psi_{\theta}|^p}{\dyle \int_{\Omega} f(x) \phi_{\theta}\, dx }:=\mu^*,
$$
and the result follows.
\end{proof}
\section{Existence result}\label{4}

In this section we consider the problem
\begin{equation}\label{existencia}
\left\{
\begin{array}{rcll}
(-\Delta)^s u &= &\lambda\dfrac{u}{|x|^{2s}} +|\nabla u|^{p}+\mu f & \mbox{ in }\Omega , \\ u &=& 0 &\hbox{  in } \mathbb{R}^N\setminus\Omega,\\ u&>&0 &\hbox{ in }\Omega,
\end{array}
\right.
\end{equation}
where $0<\lambda<\Lambda_{N,s}$, $\mu>0$, $s\in (\frac 12,1)$, $p<p_{+} (\lambda,s)<2s$ and $f\in L^\sigma(\O)$ for some convenient $\s>1$. The main goal of this section is to show that, under additional hypotheses
on $f$, we are able to build a suitable supersolution and then by a monotonicity argument,  to
prove the existence of a minimal positive solution.

\

Before the statement of  the  existence result of this section, let us recall a
compactness result obtained in \cite{CV2} and the  comparison result that will
be used in this section.
\begin{Theorem}\label{key}
Suppose that $s\in (\frac 12, 1)$ and let $f\in \mathfrak{M}(\O)$, a Radon measure. Then the problem
\begin{equation}\label{gener}
\left\{
\begin{array}{rcll}
(-\Delta)^s v &= & f &
\mbox{ in }\Omega , \\ u &=& 0 &\hbox{  in } \mathbb{R}^N\setminus\Omega,
\end{array}%
\right.
\end{equation}
has a unique weak solution in the sense of Definition \ref{def1}  such that,
\begin{enumerate}
\item  $|\n v|\in M^{p_*,\infty}(\O)$, the Marcinkiewicz space,  with
\begin{equation}\label{p*sub} p_*=\frac{N}{N-2s+1}\end{equation}
and as a consequence
 $v\in W^{1,q}_0(\Omega)$ for all $q<p_*$. Moreover
\begin{equation}\label{dd}
||v||_{W^{1,q}_0(\Omega)}\le C(N,q,\O)||f||_{\mathfrak{M}(\O)}.
\end{equation}
\item For $f\in L^1(\O)$, setting\; $T: L^1(\O)\to  W^{1,q}_0(\Omega)$, with $T(f)=v$, then $T$ is a compact operator.
\end{enumerate}
\end{Theorem}

The next comparison principle is proved in \cite{APN}, and extends the one proved in \cite{APi} in the local case.
\begin{Theorem}\label{compa2}(Comparison Principle).
Let $g\in L^1(\O)$ be a nonnegative function and consider $p_*$ defined in \eqref{p*sub}.  Assume that for all $\xi_1,\xi_2\in \mathbb{R}^N$,
$$H:\Omega\times \mathbb{R}^N\rightarrow\mathbb{R}^+ \hbox{  satisfies   } |H(x,\xi_1)-H(x,\xi_2)|\le C b(x)|\xi_1-\xi_2|$$
where $b\in L^\s(\O)$ for some $\s>\frac{N}{2s-1}$. Consider  $w_1, w_2$ two positive functions such that $w_1, w_2\in W^{1,p}(\O)$ for all $p<p_*$, $(-\Delta)^s
w_1, (-\Delta)^s  w_2\in L^1(\O)$,  $w_1\le w_2$ in $\ren \setminus\Omega$ and
\begin{equation}\label{eq:compar1} \left\{
\begin{array}{rcll}
(-\Delta)^s  w_1 &\le &  H(x,\nabla w_1)+g & \hbox{ in  }\Omega,\\ (-\Delta)^s w_2 & \ge & H(x,\nabla w_2)+g & \hbox{ in  }\Omega.
\end{array} \right.
\end{equation}
Then, $w_2\ge w_1$ in $\Omega$.
\end{Theorem}

As a consequence, we have the following
\begin{Theorem}\label{comparacion}
Assume that $g\in L^1(\O)$ is a nonnegative function. Let $w_1, w_2$ be two nonnegative functions such that $w_1, w_2\in W^{1,p}(\O)$ for some $1\le q<p_*$,
$(-\Delta)^s  w_1, (-\Delta)^s  w_2\in L^1(\O)$,  $w_1\le w_2$ in $\ren \setminus\Omega$ and
\begin{equation}\label{comparacion1} \left\{
\begin{array}{rcll}
(-\Delta)^s  w_1 &\le &  |\n w_1|^q+g & \hbox{ in  }\Omega,\\ (-\Delta)^s w_2 & \ge & |\n w_2|^q+g & \hbox{ in  }\Omega.
\end{array} \right.
\end{equation}
Then, $w_2\ge w_1$ in $\ren$.
\end{Theorem}

Since we will use the representation formula for the solution of problem
\eqref{gener}, then we recall  the main properties of the Green function
associated to the fractional laplacian. The proof can be found in  \cite{BJ0},
\cite{BJ1} \cite{BJ}, using a probabilistic approach.
\begin{Lemma}\label{estimmm}
Let $\mathcal{G}_s$ be the Green kernel of $(-\Delta)^s$ and suppose that $s\in (\frac 12,1)$, then
\begin{equation}\label{green1}
\mathcal{G}_s(x,y)\le C_1\min\{\frac{1}{|x-y|^{N-2s}}, \frac{\d^s(x)}{|x-y|^{N-s}}, \frac{\d^s(y)}{|x-y|^{N-s}}\},
\end{equation}
and
\begin{equation}\label{green2}
|\n_x \mathcal{G}_s(x,y)|\le C_2 \mathcal{G}_s(x,y)\max\{\frac{1}{|x-y|}, \frac{1}{\d(x)}\}.
\end{equation}
\end{Lemma}

\subsection{ A radial supersolution}
We will start by building  a radial  supersolution with an appropriate regularity.

We begin with the case $p_-(\l,s)<q<p_{+}(\lambda,s)<2s$. Let $w_1$ be the solution to the equation \eqref{soulrn} obtained in the first section. Recall that
$w_1(x)=\dfrac{A}{|x|^{\theta_0}}$ with $\theta_0=\frac{N-2s}{2}-\beta$ and

$$
(-\D)^s w_1(x)-\l \frac{w_1}{|x|^{2s}}=\frac{A(\gamma_\beta-\l)}{|x|^{\theta_0+2s}}= \dfrac {A^{p-1} |\theta_0|^p}{|x|^{\theta_0+2s}}.
$$
By the definition of {$\gamma_{\beta}$} given in \eqref{gamma}, it holds that $(\g_{\beta}-\l)>0$ if and only if $\theta_0\in (\mu(\l), \bar\mu(\l))$. It is clear that, in order to get
$$
\frac{A(\gamma_{\beta}-\l)}{|x|^{\theta_0+2s}}\ge |\n w_1|^p  \mbox{  in   }\O,
$$
we need that $\theta_0<p_{+}(\lambda,s)$.

Now, fix $\theta \in (\mu(\l), \bar\mu(\l))$ close to $\mu(\l)$ such that if we set
$w(x)= A |x|^{-\theta},\quad A>0$, then

 $$
(-\D)^s w(x)=\l \frac{w}{|x|^{2s}}+ \frac{C(A,\l)}{|x|^{2s+\theta}}.
$$

It is clear that $|\n w(x)|=\dfrac{A\, \theta}{|x|^{\theta+1}}$, hence
$\dfrac{C(A,\l)}{|x|^{2s+\theta}}\ge |\n w(x)|^p$ in a neighborhood of the origin if $\theta+2s>q(\theta+1)$. Thus $\theta< \frac{2s-p}{p-1}$.

Hence we can fix $\a>0$ such that $\theta<\frac{2s-\a}{\a-1}$ and $p_-(\l,s)<p<\a<p_{+}(\lambda,s)$.  From now on, we fix $\a$ such that the above construction holds.

Notice that, since $p<p_{+}(\lambda,s)$, then  $\frac{2s-p}{p-1}>\mu(\l)$. Also, since $p_-(\l,s)<p$, then
$\frac{2s-p}{p-1}<\bar\mu(\l)$.

Clearly, if $f\le \dfrac{1}{|x|^{2s+\theta}}$, then $w_1$ is a supersolution to problem \eqref{existencia} for $\mu<\mu^*$.

\

We analyze now some properties of this supersolution.

Recall that $\mathcal{G}_s$ is the Green kernel of $(-\Delta)^s$ and define
$$
K(y)=\io \dfrac{w^{\a-1}(x)\mathcal{G}_s(x,y)}{|x-y|^{\a}}dx,
$$
where $\a<2s$ to be chosen later. We claim that $K\in L^\infty(\O)$. To show the claim, we observe that

$$K(y)\le C\io \dfrac{1}{|x|^{\theta(\a-1)}|x-y|^{N-2s+\a}}dx.$$
Since $\theta<\frac{2s-\a}{\a-1}$ and $\a<2s$, it holds that $\frac{N}{2s-\a}<\frac{N}{\theta(\a-1)}$. Hence we get the existence of $\frac{N}{2s-\a}<\s<\frac{N}{\theta(\a-1)}$ such that $\dfrac{1}{|x|^{\theta(\a-1)}}\in L^\s(\O)$.

Using H\"older inequality, we have that
$$K(y)\le C\Big(\io \dfrac{1}{|x|^{\theta(\a-1)}})^\sigma dx\Big)^{\frac{1}{\s}}
\Big(\io \dfrac{1}{|x-y|^{\s'(N-2s+\a)}}dx\Big)^{\frac{1}{\s'}}.
$$
Since $\s'(N-2s+\a)<N$, then $K(y)\le C$ for all $y\in \O$ and the claim follows.

Consider now $\psi$ to be the unique solution to the problem
\begin{equation}\label{inter00} \left\{
\begin{array}{rcll}
(-\Delta)^s \psi &= & \dfrac{w^{\a-1}(x)}{\d^{\a}(x)}=\dfrac{{A^{\alpha-1}}}{|x|^{\theta(\a-1)}\d^{\a}(x)} & \text{   in }\Omega , \\  \psi &=& 0 & \hbox{  in } \mathbb{R}^N\setminus\Omega.
\end{array} \right.
\end{equation}
Since $\theta(\a-1)<2s$,  as in \cite{APN}, we  can prove that $\psi\in L^\infty(\O)$.

We are now in position to state the main existence result.
\begin{Theorem}\label{exit1}
Assume $f\in L^\infty(\Omega)$ and suppose that  $1<p<p_{+}(\lambda,s)<2s$. Then problem \eqref{existencia} has a
solution $u$ such that $u\in W^{1,p}_0(\O)$.
\end{Theorem}
\begin{proof}
 We divide the proof into two parts according to the value of $p$.

{\bf First case: $p_{-}(\lambda,s)<p<p_{+}(\lambda,s)$.}

Let $w$ be the supersolution obtained in the previous computation, then $w\in W^{1,\a}(\O)$. Consider $u_n$ to be the unique solution to the approximating problem
\begin{equation}\label{aprox1} \left\{
\begin{array}{rcll}
(-\Delta)^s u_n &= & \dfrac{|\n u_n|^p}{1+\frac 1n |\n u_n|^p} +\lambda \dfrac{u_n}{|x|^{2s}}+\mu f & \text{   in }\Omega , \\  u_n &=& 0 &\hbox{  in }\mathbb{R}^N\setminus\Omega.
\end{array} \right.
\end{equation}
By the comparison principle in Theorem \ref{compa2}, it follows that $u_n\le u_{n+1}\le w$ for all $n$. Since $w\in L^{p^*}(\O)$, then there exists $u$ such that $u_n\uparrow u$
strongly in $L^{p^*}(\O)$. Define
$$
g_n(x)=\dfrac{|\n u_n|^p}{1+\frac 1n |\n u_n|^p}+\l \frac{u_n}{|x|^{2s}} +\mu f,
$$
since $u_n\le w$, using the positive first eigenfunction of the fractional laplacian $\varphi_1$ as test function in \eqref{aprox1} and using the fact that $\varphi_1\backsimeq \delta^s$, it holds that
$$
\io g_n(x)\d^s(x) dx \le \l_1\io u_n \varphi_1\le C(\O)\io w \d^s(x) \,dx\le C\mbox{  for all   }n.
$$
We claim that the sequence $\{u_n\}_n$ is bounded in $W^{1,\a}_0(\O)$ where $\a<2s$ is chosen as in the definition of the supersolution.

We follow  the same ideas as in  \cite{APN}. We
have that
$$
u_n(x)=\io \mathcal{G}_s(x,y) g_n(y)dy.
$$
Hence
$$
|\n u_n(x)|\le\io |\n_x \mathcal{G}_s(x,y)|g_n(y)dy.
$$
Fix $1<\a<2s$ and define $h(x,y)=\max\bigg\{\dfrac{1}{|x-y|},\dfrac{1}{\d(x)}\bigg\}$. Then,
\begin{eqnarray*}
|\n u_n(x)|^{\a} &\le & \Big(\io |\n_x \mathcal{G}_s(x,y)|g_n(y)dy\Big)^{\a}{ \le} \Big(\io h(x,y) \mathcal{G}_s(x,y)g_n(y)dy\Big)^{\a}\\& \le & \dyle \Big(\io
(h(x,y))^{\a} \mathcal{G}_s(x,y)g_n(y)dy\Big)\Big(\io \mathcal{G}_s(x,y)g_n(y)dy\Big)^{\a-1}\\ &\le & \dyle \Big(\io (h^{\a}(x,y)
\mathcal{G}_s(x,y)g_n(y)dy\Big)u^{\a-1}_n(x)\\ &\le & \dyle \io \Big(h^{\a}(x,y) \mathcal{G}_s(x,y)g_n(y)dy\Big)w^{\a-1}(x)\\ &\le & \dyle \io \Big(h^{\a}(x,y)
\mathcal{G}_s(x,y)|\n u_n(y)|^pdy\Big)w^{\a-1}(x)+
\l \io \Big(h^{\a}(x,y) \mathcal{G}_s(x,y)\frac{u_n(y)}{|y|^{2s}}dy\Big)w^{\a-1}(x)\\
&+ & \mu \dyle \io \Big(h^{\a}(x,y) \mathcal{G}_s(x,y)f(y)dy\Big)w^{\a-1}(x).
\end{eqnarray*}
Thus
\begin{eqnarray*}
\dyle\io |\n u_n|^{\a}dx & \le & \dyle \io|\n u_n(y)|^p \Big(\io h^{\a}(x,y) \mathcal{G}_s(x,y)w^{\a-1}(x)\Big)dy\\
& + & \dyle \l \io \frac{w(y)}{|y|^{2s}}\Big(\io h^{\a}(x,y)\mathcal{G}_s(x,y)w^{\a-1}(x)dx\big)dy\\
& + & \dyle \mu \io f(y)\Big(\io
h^{\a}(x,y)\mathcal{G}_s(x,y)w^{\a-1}(x)dx\big)dy\\ &\equiv & J_1+J_2+J_3.
\end{eqnarray*}
Observing that $h^\alpha(x,y)\le \dfrac{1}{|x-y|^\alpha}+\dfrac{1}{\d(x)^\alpha}$, then
\begin{eqnarray*}
J_1 & \le & \dyle \io|\n u_n(y)|^p \Big(\io \dfrac{w^{\a-1}(x)\mathcal{G}_s(x,y)}{|x-y|^{\a}}dx\Big)dy+
\io|\n u_n(y)|^p \Big(\io \dfrac{w^{\a-1}(x)\mathcal{G}_s(x,y)}{\d^{\a}(x)}dx\Big)dy\\
& \le & \dyle \io|\n u_n(y)|^p \Big(\io \dfrac{w^{\a-1}(x)\mathcal{G}_s(x,y)}{|x-y|^{\a}}dx\Big)dy+
\io|\n u_n(y)|^p \psi(y)dy.
\end{eqnarray*}
By using the hypothesis on $w$, we reach that
\begin{eqnarray*}
J_1 &\le & C\io|\n u_n(y)|^p dy+\Big(\io|\n u_n(y)|^\a dy\Big)^{\frac{p}{\a}}\Big(\io \psi^{\frac{\a}{p-\a}}dy\Big)^{\frac{\a-p}{\a}}\\
&\le & \dyle C_1\io|\n u_n(y)|^p dy+C_2\Big(\io|\n u_n(y)|^\a dy\Big)^{\frac{p}{\a}}.
\end{eqnarray*}

We deal now with $J_2$.

\begin{eqnarray*}
J_2 & \le & \dyle \io \frac{w(y)}{|y|^{2s}} \Big(\io \dfrac{w^{\a-1}(x)\mathcal{G}_s(x,y)}{|x-y|^{\a}}dx\Big)dy+
\io \frac{w(y)}{|y|^{2s}} \Big(\io \dfrac{w^{\a-1}(x)\mathcal{G}_s(x,y)}{\d^{\a}(x)}dx\Big)dy\\
& \le & \dyle \io \frac{w(y)}{|y|^{2s}}\Big(\io \dfrac{w^{\a-1}(x)\mathcal{G}_s(x,y)}{|x-y|^{\a}}dx\Big)dy+
\io \frac{w(y)}{|y|^{2s}}\psi(y)dy\le C\dyle \io \frac{w(y)}{|y|^{2s}} dy\le C.
\end{eqnarray*}

For $J_3$, we have
\begin{eqnarray*}
J_3 & \le & \dyle \io f(y)\Big(\io \dfrac{w^{\a-1}(x)\mathcal{G}_s(x,y)}{|x-y|^{\a}}dx\Big)dy+
\io f(y)\psi(y)dy.
\end{eqnarray*}
Hence,
\begin{eqnarray*}
J_3 &\le & C\io f(y)dy+\Big(\io f^{\frac{N}{p'(2s-1)}}dy\Big)^{\frac{p'(2s-1)}{N}}\Big(\io \psi^{\frac{N}{N-p'(2s-1)}}
dy\Big)^{\frac{N-p'(2s-1)}{N}}\\
&\le & C.
\end{eqnarray*}
Therefore we conclude that
$$
\io|\n u_n(x)|^\a dx\le C_1\io|\n u_n(x)|^pdx +C_2.
$$
Choosing $\a>q$ and by H\"older inequality,  we obtain that
$$\io|\n u_n(x)|^\a dx\le C \hbox{ for all  } n.$$

As a consequence we get that the sequence $\{g_n\}_n$ is bounded in $L^{1+\e}(\O)$ for some $\e>0$. By the compactness result in Proposition \ref{key}, we obtain that, up to a subsequence, $u_n\to u$ strongly in $W^{1,r}_0(\Omega)$ for all $r<p_*$ and $|\n u_n|\to |\n u|$ a.e.  in $\O$. Hence by Vitali lemma we reach that $u_n\to u$
strongly in $W^{1,\a}_0(\Omega)$ with $\a$ previously chosen. Since $p<\a$, then
$$
\dfrac{|\n u_n|^p}{1+\frac 1n |\n u_n|^p}\to |\n u|^p\mbox{  strongly  in   }L^1(\O).
$$
Hence, $u$ is a solution to \eqref{existencia} with $u\in W^{1,\a}_0(\Omega)$.

{\bf Second case: $1<p\le p_{-}(\lambda,s)$.}
We begin by proving that problem \eqref{soulrn} has a supersolution in a small ball $B_r(0)$ that enjoys the same regularity properties as $w$.

Fix $p_0\in (p_{-}(\lambda,s),p_{+}(\lambda,s))$ and consider $u_{p_0}$ the solution to problem \eqref{existencia} obtained in the first case with $f\equiv \frac{1}{|x|^{2s}}$. Since $p_0>p$, then for all $\e>0$ there exists $C(\e)>0$ such that for all $\s\ge 0$,
$\s^{p_0}\ge C(\e)\s^p-\e$. Thus
$$
|\n u_{p_0}|^{p_0}\ge C(\e)|\n u_{p_0}|^p-\e.
$$
Therefore,
$$
(-\Delta)^s u_{p_0}\ge \lambda \dfrac{u_{p_0}}{|x|^{2s}}+C(\e)|\n u_{p_0}|^p-\e+\frac{\mu}{|x|^{2s}} \mbox{  in  }\O.
$$
Let fix $r>0$ small enough such that $\dfrac{\mu}{|x|^{2s}}-\e\ge \dfrac{\mu_0}{|x|^{2s}}$ in $B_r(0)$. Hence we conclude that $u_{p_0}$ satisfies
$$
(-\Delta)^s u_{p_0}\ge \lambda \dfrac{u_{p_0}}{|x|^{2s}}+C(\e)|\n u_{p_0}|^p+\frac{\mu_0}{|x|^{2s}} \mbox{  in  }B_r(0).
$$
Setting $u_p=C(\e) u_{p_0}$, we reach that
$$
(-\Delta)^s u_{p}\ge \lambda \dfrac{u_{p}}{|x|^{2s}}+|\n u_{p}|^p+\frac{\mu_1}{|x|^{2s}} \mbox{  in  }B_r(0).
$$
Thus $u_p$ is a supersolution to problem \eqref{existencia} with the same regularity properties as $w$. Hence the existence result follows using the same approach as in the first case.

Notice that if $u$ is a supersolution in $B_r(0)$, then for $x\in B_R(0)$ with $R>r$ and by setting $\hat{u}(x)=u(\frac{r}{R}x)$, then there exists a constant $C:=C(R,r,p)$ such that $\check{u}:=C\hat{u}$ is a supersolution to \eqref{existencia} in $B_R(0)$ with $\mu:=\hat{\mu}$.

Now we consider the case of general domain $\O$. Let $R>1$ be such that $\O\subset \subset B_R(0)$. It is clear that $\check{u}$ is a supersolution to \eqref{existencia} that has the same properties of $w$. Hence we conclude the proof of Theorem \ref{exit1}.
\end{proof}

\begin{Corollary} Assume that $p<p_{+}(\lambda,s)$ and that $f\le \dfrac{C}{|x|^{2s}}$ with $f\gneq 0$. Then there exists $\mu^*>0$ such that for all $\mu<\mu^*$, problem \eqref{existencia} has a positive solution $u$ such that $u\in W^{1,\a}_0(\Omega)$ for all $\a<2s$.
\end{Corollary}
\begin{remark}
Under the extra assumption $p<p_*=\frac{N}{N-2s+1}$, we are able to prove the existence of a solution for all $f\in L^1(\O)$ {that satisfies a suitable integrability condition near the origin}. More precisely, fixed $\l<\Lambda_{N,s}$, the problem
\begin{equation}\label{hardyff} \left\{
\begin{array}{rcll}
(-\Delta)^s v &= & \lambda \dfrac{v}{|x|^{2s}}+f & \text{   in }\Omega , \\  \phi &=& 0 &\hbox{  in }\mathbb{R}^N\setminus\Omega,
\end{array} \right.
\end{equation}
has a weak solution if and only if $\dint_\Omega f|x|^{-\mu(\l)}dx<\infty$ (see \cite{AMPP}).

Moreover we have
$$\io v |x|^{-\mu(\l)} dx<\infty,$$
and
\begin{equation}\label{fix2}
||v||_{L^1(|x|^{-\mu(\l)}dx,\O)}+\bigg\|\frac{v}{|x|^{2s}}\bigg\|_{L^1(\O)}+ \bigg\|\n v|\bigg\|_{L^\s(\O)}\le C(\O,\l)||f||_{L^1(|x|^{-\mu(\l)}dx,\O)}, \mbox{  for all  }\s<\frac{N}{N-2s+1}.
\end{equation}
We refer to \cite{AMPP} for the proof.
\end{remark}
Suppose now that $f\in L^1(|x|^{-\mu(\l)-a_0}dx,\O)$,
hence there exists $\l_1\in (\l,\Lambda_{N,s})$ such that $\mu(\l_1)=\mu(\l)+a_0$. Define $\psi$ to be the unique solution to problem
\begin{equation}\label{hardyffa} \left\{
\begin{array}{rcll}
(-\Delta)^s \psi &= & \l_1 \dfrac{\psi}{|x|^{2s}}+1 & \text{   in }\Omega , \\  \psi &=& 0 &\hbox{  in }\mathbb{R}^N\setminus\Omega,
\end{array} \right.
\end{equation}
then $\psi\simeq |x|^{-\mu(\l)-a_0}$ near the origin. It is clear also that $\psi\in L^\infty(\Omega\backslash B_r(0))$.

Using $\psi$ as a test function in problem \eqref{hardyff}, it holds that
$$
(\l_1-\l)\io \frac{v \psi }{|x|^{2s}}dx \le \io f\psi dx.
$$
Hence
\begin{equation}\label{regumas}
\io \frac{v}{|x|^{2s+\mu(\l)}}dx \le C(\O\,\l,a_0)||f||_{L^1(|x|^{-\mu(\l)-a_0}dx,\O)}.
\end{equation}
The next proposition will be the key in order to show the existence of a solution to problem \eqref{existencia} under the above general hypothesis on $f$.
\begin{Proposition}\label{main-exi}
Assume that $f\in L^1(|x|^{-\mu(\l)-a_0}dx,\O)$ for some $a_0>0$ and $v$ to be the unique weak solution to problem \eqref{hardyff}, then
\begin{equation}\label{fix3}
\bigg\|\n v\bigg\|_{L^\a(|x|^{-\mu(\l)}dx,\O)}\le C(\O,\l,a_0)||f||_{L^1(|x|^{-\mu(\l)-a_0}dx,\O)}\mbox{  for all  }\a<\frac{N}{N-2s+1}.
\end{equation}
\end{Proposition}
\begin{proof}
Notice that
$$
\bigg\|\n v\bigg\|_{L^\a(\O)}\le C(\O,\l,a_0)||f||_{L^1(|x|^{-\mu(\l)-a_0}dx,\O)}\mbox{  for all  }\a<\frac{N}{N-2s+1}.
$$
Hence to prove the claim we have just to show that
$$
\int_{B_r(0)}|\n v|^\a|x|^{-\mu(\l)} dx\le C(\O,\l,a_0)||f||_{L^1(|x|^{-\mu(\l)-a_0}dx,\O)} \mbox{  for all  }\a<\frac{N}{N-2s+1}.
$$
We set $g(x):=\lambda \dfrac{v}{|x|^{2s}}+\mu f$, then $v(x)=\io\mathcal{G}_s(x,y)g(y)dy$. Hence
$$
|\n v(x)|\le\io |\n_x \mathcal{G}_s(x,y)|g(y)dy.
$$
Fix $1<\a<p_*=\frac{N}{N-2s+1}$ and define $h(x,y)=\max\bigg\{\dfrac{1}{|x-y|},\dfrac{1}{\d(x)}\bigg\}$. Then,
\begin{eqnarray*}
|\n v(x)|^{\a} &\le & \Big(\io |\n_x \mathcal{G}_s(x,y)|g(y)dy\Big)^{\a}{ \le} \Big(\io h(x,y) \mathcal{G}_s(x,y)g_n(y)dy\Big)^{\a}\\& \le & \dyle \Big(\io
(h(x,y))^{\a} \mathcal{G}_s(x,y)g(y)dy\Big)\Big(\io \mathcal{G}_s(x,y)g(y)dy\Big)^{\a-1}\\ &\le & \dyle \io \Big(h^{\a}(x,y) \mathcal{G}_s(x,y)g(y)dy\Big)v^{\a-1}(x)\\ &\le & \dyle \l \io \Big(h^{\a}(x,y) \mathcal{G}_s(x,y)\frac{v(y)}{|y|^{2s}}dy\Big)v^{\a-1}(x)+\mu \dyle \io \Big(h^{\a}(x,y) \mathcal{G}_s(x,y)f(y)dy\Big)v^{\a-1}(x).
\end{eqnarray*}
Thus
\begin{eqnarray*}
\dyle\int_{B_r(0)} |\n v|^{\a} |x|^{-\mu(\l)} dx & \le & \l \io \frac{v(y)}{|y|^{2s}}\Big(\int_{B_r(0)} h^{\a}(x,y)\mathcal{G}_s(x,y)v^{\a-1}(x)|x|^{-\mu(\l)} dx\big)dy\\
& + & \dyle \mu \io f(y)\Big(\int_{B_r(0)}
h^{\a}(x,y)\mathcal{G}_s(x,y)v^{\a-1}(x)|x|^{-\mu(\l)} dx\big)dy\\ &\equiv & J_1+J_2.
\end{eqnarray*}
Recall that $h(x,y)=\max\{\dfrac{1}{|x-y|},\dfrac{1}{\d(x)}\}$. then for all $x\in B_r(0)\subset\subset \O$, we have
$$
\dfrac{C_1(\O,B_r(0))}{|x-y|}\le h(x,y)\le \dfrac{C_2(\O,B_r(0))}{|x-y|}.$$
Let us begin by estimating $J_1$.
Recall that, by \eqref{regumas}, we have
$$
\io \frac{v(y)}{|y|^{2s+\mu(\l)}}dy \le C \io \frac{f(y)}{|y|^{\mu(\l)+a_0}}dy.
$$
Therefore we obtain that
\begin{eqnarray*}
J_1 & \le & \dyle C_2\io \frac{v(y)}{|y|^{2s}} \Big(\int_{B_r(0)} \dfrac{v^{\a-1}(x)\mathcal{G}_s(x,y)}{|x|^{\mu(\l)}|x-y|^{\a}}dx\Big)dy\\&\le& C_2\io \frac{v(y)}{|y|^{2s}} \Big(\int_{B_r(0)} \dfrac{v^{\a-1}(x)}{|x|^{\mu(\l)}|x-y|^{N-(2s-\a)}}dx\Big)dy\\
& \le & \dyle C\io \frac{v(y)}{|y|^{2s}}\Big(\int_{B_r(0)\cap \{|x|\ge \frac 12 |y|\}} \dfrac{v^{\a-1}(x)}{|x|^{\mu(\l)}|x-y|^{N-(2s-\a)}}dx\Big)dy\\
&+ & \dyle C\io \frac{v(y)}{|y|^{2s}}\Big(\int_{B_r(0)\cap \{|x|<\frac 12 |y|\}} \dfrac{v^{\a-1}(x)}{|x|^{\mu(\l)}|x-y|^{N-(2s-\a)}}dx\Big)dy\\
&\le & J_{11}+J_{12}.
\end{eqnarray*}
To estimate $J_{11}$, we have
\begin{eqnarray*}
J_{11} & \le & C\io \frac{v(y)}{|y|^{2s+\mu(\l)}}\Big(\int_{B_r(0)}\dfrac{v^{\a-1}(x)}{|x-y|^{N-(2s-\a)}}dx\Big)dy.
\end{eqnarray*}
Recall that $v\in L^\s(\O)$ for all $\s<p_2=\frac{N}{N-2s}$. Since $\a<p_*=\frac{N}{N-2s+1}$. Let $\s_0<p_2$ and using H\"older inequality we obtain that
$$
\int_{B_r(0)}\dfrac{v^{\a-1}(x)}{|x-y|^{N-(2s-\a)}}dx\le \bigg(\int_{B_r(0)}v^{\s_0}dx \bigg)^{\frac{\a-1}{\s_0}}
\bigg(\int_{B_r(0)}\frac{1}{|x-y|^{\frac{(N-(2s-\a))\s_0}{\s_0-(\a-1)}}}dx\bigg)^{\frac{\s_0-(\a-1)}{\s_0}}.
$$
Since $\a<p_*$, then we can chose $\s_0$ close to $p_2$ such that $\frac{(N-(2s-\a))\s_0}{\s_0-(\a-1)}<N$.
Thus
$$
\int_{B_r(0)}\frac{1}{|x-y|^{\frac{(N-(2s-\a))\s_0}{\s_0-(\a-1)}}}dx\le C(r,\O),
$$
and then
\begin{equation}\label{j11}
J_{11}\le C \bigg(\int_{B_r(0)}v^{\s_0}dx \bigg)^{\frac{\a-1}{\s_0}}\bigg(\io \frac{v(y)}{|y|^{2s+\mu(\l)}}dy\bigg)\le C\bigg(\io \frac{f(y)}{|y|^{\mu(\l)+a_0}}dy\bigg)^\a.
\end{equation}
We deal now with $J_{12}$. Notice that $\{|x|\le \frac 12 |y|\}\subset \{|x-y|\ge \frac 12 |y|\}$. Thus
\begin{eqnarray*}
J_{12} & \le & C\io \frac{v(y)}{|y|^{\mu(\l)+2s}}\Big(\int_{B_r(0)}\dfrac{v^{\a-1}(x)}{|x|^{\mu(\l)}|x-y|^{N-(2s+\mu(\l)-\a)}}dx\Big)dy.
\end{eqnarray*}
As in the estimate of $J_{11}$, setting $\theta=\dfrac{\s_0}{\s_0-(\a-1)}$, we have
\begin{eqnarray*}
& \dyle \int_{B_r(0)}\dfrac{v^{\a-1}(x)}{|x|^{\mu(\l)}|x-y|^{N-(2s+\mu(\l)-\a)}}dx\\
&\dyle \le \bigg(\int_{B_r(0)}v^{\s_0}dx \bigg)^{\frac{\a-1}{\s_0}}
\bigg(\int_{B_r(0)}\frac{1}{|x|^{\mu(\l)\theta}|x-y|^{(N-(2s+\mu(\l)-\a))\theta}}dx\bigg)^{\frac{1}{\theta}}.
\end{eqnarray*}
Since $\mu(\l)\theta<N$( for $\s_0$ close to $p_2$), using again H\"older inequality, we obtain that
\begin{eqnarray*}
& \dyle \int_{B_r(0)}\frac{1}{|x|^{\mu(\l)\theta}|x-y|^{(N-(2s+\mu(\l)-\a))\theta}}dx \le \\
& \dyle \bigg(\int_{B_r(0)}\frac{1}{|x|^{N-\e}}dx\bigg)^{\frac{\mu(\l)\theta}{N-\e}}  \bigg(\int_{B_r(0)}
\frac{1}{|x-y|^{\frac{(N-(2s+\mu(\l)-\a))\theta (N-\e)}{N-\e-\mu(\l)\theta}}}dx\bigg)^{\frac{N-\e-\mu(\l)\theta}{N-\e}}.
\end{eqnarray*}
By a direct computation and using the fact that $\a<p_*$, we obtain that $\frac{(N-(2s+\mu(\l)-\a))\theta N}{N-\mu(\l)\theta}<N$. Hence we get the existence of $\e>0$ small such that $\frac{(N-(2s+\mu(\l)-\a))\theta (N-\e)}{N-\e-\mu(\l)\theta}<N$ and we conclude that
\begin{equation}\label{j12}
J_{12}\le C\bigg(\io \frac{f(y)}{|y|^{\mu(\l)+a_0}}dy\bigg)^\a.
\end{equation}
As a consequence, we have
\begin{equation}\label{j1}
J_{1}\le C\bigg(\io \frac{f(y)}{|y|^{\mu(\l)+a_0}}dy\bigg)^\a.
\end{equation}
We deal now with $J_2$. We will use the same decomposition as in the estimate of $J_1$.
\begin{eqnarray*}
J_2 & \le & C\io f(y)\Big(\int_{B_r(0)} \dfrac{v^{\a-1}(x)}{|x|^{\mu(\l)}|x-y|^{N-(2s-\a)}}dx\Big)dy\\
& \le & \dyle C\io f(y)\Big(\int_{B_r(0)\cap \{|x|\ge \frac 12 |y|\}} \dfrac{v^{\a-1}(x)}{|x|^{\mu(\l)}|x-y|^{N-(2s-\a)}}dx\Big)dy\\
&+ & \dyle C_3\io f(y)\Big(\int_{B_r(0)\cap \{|x|<\frac 12 |y|\}} \dfrac{v^{\a-1}(x)}{|x|^{\mu(\l)}|x-y|^{N-(2s-\a)}}dx\Big)dy\\
&\le & J_{21}+J_{22}.
\end{eqnarray*}
To estimate $J_{21}$, we have
\begin{eqnarray*}
J_{21} & \le & C_3\io \frac{f(y)}{|y|^{\mu(\l)}}\Big(\int_{B_r(0)}\dfrac{v^{\a-1}(x)}{|x-y|^{N-(2s-\a)}}dx\Big)dy.
\end{eqnarray*}
For the integral $\int_{B_r(0)}\dfrac{v^{\a-1}(x)}{|x-y|^{N-(2s-\a)}}dx$, we use the same computations as in the estimate of $J_{11}$ (since we are with the same range of  parameters), and then we conclude that
$$
\begin{array}{rcl}
\displaystyle\int_{B_r(0)}\dfrac{v^{\a-1}(x)}{|x-y|^{N-(2s-\a)}}dx&\le& \bigg(\displaystyle\int_{B_r(0)}v^{\s_0}dx \bigg)^{\frac{\a-1}{\s_0}}
\bigg(\displaystyle\int_{B_r(0)}\frac{1}{|x-y|^{\frac{(N-(2s-\a))\s_0}{\s_0-(\a-1)}}}dx\bigg)^{\frac{\s_0-(\a-1)}{\s_0}}\\
&\le& C\bigg(\displaystyle\int_{B_r(0)}v^{\s_0}dx \bigg)^{\frac{\a-1}{\s_0}}.
\end{array}
$$
Thus,
\begin{equation}\label{j21}
J_{21}\le C \bigg(\int_{B_r(0)}v^{\s_0}dx \bigg)^{\frac{\a-1}{\s_0}}\bigg(\io \frac{f(y)}{|y|^{\mu(\l)}}dy\bigg)\le C\bigg(\io \frac{f(y)}{|y|^{\mu(\l)}}dy\bigg)^\a.
\end{equation}
To analyze $J_{22}$ we use also the fact that $\{|x|\le \frac 12 |y|\}\subset \{|x-y|\ge \frac 12 |y|\}$. Then
\begin{eqnarray*}
J_{22} & \le & C\io \frac{f(y)}{|y|^{\mu(\l)}}\Big(\int_{B_r(0)}\dfrac{v^{\a-1}(x)}{|x|^{\mu(\l)}|x-y|^{N-(2s+\mu(\l)-\a)}}dx\Big)dy.
\end{eqnarray*}
As in the estimate of $J_{12}$, it holds that
\begin{eqnarray*}
& \dyle \int_{B_r(0)}\dfrac{v^{\a-1}(x)}{|x|^{\mu(\l)}|x-y|^{N-(2s+\mu(\l)-\a)}}dx\\
&\dyle \le \bigg(\int_{B_r(0)}v^{\s_0}dx \bigg)^{\frac{\a-1}{\s_0}}
\bigg(\int_{B_r(0)}\frac{1}{|x|^{\mu(\l)\theta}|x-y|^{(N-(2s+\mu(\l)-\a))\theta}}dx\bigg)^{\frac{1}{\theta}}\le C \bigg(\int_{B_r(0)}v^{\s_0}dx \bigg)^{\frac{\a-1}{\s_0}}.
\end{eqnarray*}
Therefore, we obtain
\begin{equation}\label{j22}
J_{22}\le C\bigg(\io \frac{f(y)}{|y|^{\mu(\l)}}dy\bigg)^\a.
\end{equation}
Hence,
\begin{equation}\label{j2}
J_{2}\le C\bigg(\io \frac{f(y)}{|y|^{\mu(\l)}}dy\bigg)^\a.
\end{equation}
From \eqref{j1} and \eqref{j2} it holds that
$$
\dyle\int_{B_r(0)} |\n v|^{\a} |x|^{-\mu(\l)} dx \le C \bigg(\io \frac{f(y)}{|y|^{\mu(\l)}}dy\bigg)^\a,
$$
and then result follows.
\end{proof}

With all the above machinery,  we are able to show the next existence result.
\begin{Theorem}\label{exit-l1}
Assume that $1<p<p_*$ and let $f\in L^1(\Omega)$ be a nonnegative function such that $\io f|x|^{-\mu(\l)-a_0}dx<\infty$ for some $a_0>0$. Then, there exists $\mu^*>0$ such that if $\mu<\mu^*$, then problem \eqref{existencia} has a
solution $u$ such that $u\in W^{1,\s}_0(\O)$ for all $\s<\frac{N}{N-2s+1}$, moveover $\io |\n u|^p|x|^{-\mu(\l)} dx<\infty$.
\end{Theorem}
\begin{proof}
We follow again the arguments used in \cite{APN}.  Fix $p<p_*$ and let $f\in L^1(\O)$ be a nonnegative function with  $\io f|x|^{-\mu(\l)-a_0}dx<\infty$.

Fix $1<p<r<p_*$. Then, we can chose $\mu^*>0$ such that for some $l>0$, we have
$$
C_0(l+\m^*||f||_{L^1(|x|^{-\mu(\l)}dx,\O)})=l^{\frac{1}{p}},
$$
where $C_0$ is a positive constant depending only on $\O,\l$ and $C(\O,\l)$ given in \eqref{fix2}.

Let $\mu<\mu^*$ be fixed and  define the set
\begin{equation}\label{sett}
E=\{v\in W^{1,1}_0(\O): v\in W^{1,r}_0(|x|^{-\mu(\l)}dx,\O)\mbox{  and  }||\n v||_{{L^r(|x|^{-\mu(\l)}dx,\O)}}\le l^{\frac{1}{p}}\},
\end{equation}
where $p<r<p_*$. It is clear that $E$ is a closed convex set of $W^{1,1}_0(\O)$. Consider the operator
$$
\begin{array}{rcl}
T:E &\rightarrow& W^{1,1}_0(\O)\\
   v&\rightarrow&T(v)=u
\end{array}
$$
where $u$ is the unique solution to problem
\begin{equation}
\left\{
\begin{array}{rcll}
(-\Delta)^s u &= & \lambda \dfrac{u}{|x|^{2s}} +|\nabla v|^{q}+\mu f &
\text{ in }\Omega , \\ u &=& 0 &\hbox{  in } \mathbb{R}^N\setminus\Omega,\\ u&>&0 &\hbox{
in }\Omega.
\end{array}%
\right.  \label{fix1}
\end{equation}%
Taking into consideration the definition of $E$, it holds that $|\nabla v|^{q}+\mu f \in L^1(|x|^{-\mu(\l)}dx,\O)$. Hence the existence and the uniqueness of $u$ follows using the result of \cite{AMPP} with $u\in W^{1,\s}_0(\O)$ for all $\s<\frac{N}{N-2s+1}$. Thus $T$ is well defined.

We claim that $T(E)\subset E$.  Since $r>p$, then using H\"older inequality we get the existence of $\hat{a}_0>0$ such that
$$
\io |\nabla v|^{p}|x|^{-\mu(\l)-\hat{a}_0}dx\le C(\O)\bigg(\io |\nabla v|^{r}|x|^{-\mu(\l)}dx\bigg)^{\frac{p}{r}}<
\infty.
$$
Setting $\bar{a}_0=\min\{a_0, \hat{a}_0\}$, it holds that $|\n v|^p +\mu f\in L^1(|x|^{-\mu(\l)-\bar{a}}dx, \O)$. Thus by Proposition \ref{main-exi}, we reach
$$
\bigg(\io |\nabla u|^{\s}|x|^{-\mu(\l)}dx\bigg)^{\frac{1}{\s}}\le C(N,p,\bar{a})\big\||\n v|^p +\mu f\bigg\|_{L^1(|x|^{-\mu(\l)-\bar{a}}dx, \O)}.
$$
Since $v\in E$, we conclude that
\begin{eqnarray*}
\dyle \bigg(\io |\nabla u|^{\s}|x|^{-\mu(\l)}dx\bigg)^{\frac{1}{\s}} &\le & C(N,p,\bar{a})\big(
\bigg(\io |\nabla v|^{r}|x|^{-\mu(\l)}dx\bigg)^{\frac{p}{r}} +\mu ||f||_{L^1(|x|^{-\mu(\l)-a_0}dx, \O)}\bigg)\\
&\le & \dyle C(l+\mu^* ||f||_{L^1(|x|^{-\mu(\l)-a_0}dx, \O)})\le l.
\end{eqnarray*}
Choosing $\s=r$, it holds that $u\in E$.

The continuity and the compactness of $T$ follow using closely the same arguments as in \cite{APN}.

As a conclusion and using the Schauder Fixed Point Theorem as in \cite{APN}, there exists $u\in E$ such that $T(u)=u$, $u\in W^{1,p}_0(|x|^{-\mu(\l)}dx,\O)$ and, therefore, $u$ solves \eqref{existencia}.
\end{proof}

\section{Existence under the presence of a zero order term vanishing at infinity.}

In this section we consider the problem
\begin{equation}\label{gener0}
\left\{
\begin{array}{rcll}
(-\Delta )^s u &=&\lambda \dfrac{u}{|x|^{2s}}+ \dfrac{|\nabla u|^{p}}{(1+u)^\a}+ cf &\inn \Omega,\\
u&>&0 & \inn\Omega,\\
u&=&0 & \inn(\mathbb{R}^N\setminus\Omega),
\end{array}\right.
\end{equation}
where $\a>0$ and  $p<2s$. The main objective of this section is to get a relation between $\a$ and $p$ in order to get the existence of a solution for some $p>p_+(\l, s)$.

The local case was treated in \cite{AGMP} where the term $\dfrac{1}{(1+u)^\a}$ 
is replaced by $\dfrac{1}{u^\a}$. The problem in this case is strongly related 
to the \textit{porus medium} equation with Hardy potential. 
Existence result is obtained under the condition that 
$\dfrac{(\mu(\l)+1)(p-1)-1}{\mu(\l)}<\a<p$, where $\mu(\lambda)$ is defined in \eqref{g1}. The arguments used in \cite{AGMP} are based on the choose of suitable text functions and the connection between the laplacian 
operator and the gradient term through the integration by parts formula.

This approach fails in the case of the fractional Laplacian due to the nonlocal nature of the operator and the like of a direct relation between the fractional Laplacian and the gradient term. To overcome these difficulties, we will use monotony argument and the representation formula.

The main existence result of this section is the following.
\begin{Theorem}\label{super}
Assume that $\a>2s-1$. Suppose que $0\lneqq f\le \dfrac{1}{|x|^\beta}$ where $\mu(\l)<\beta<\bar{\mu}(\l)$ and $\beta$ is close to $\mu(\l)$. Then there exists $c^*>0$ such that if $c<c^*$, then problem \eqref{gener0} has a solution in the sense of Definition \ref{def1} with $\dfrac{u}{|x|^{2s}}+ \dfrac{|\nabla u|^{p}}{(1+u)^\a}\in L^1(\O)$.
\end{Theorem}
Before starting with the proof of the previous Theorem,  we state the next comparison principle.
\begin{Proposition}\label{lm:comp}
Assume that
$H:\mathbb{R}^N\rightarrow\mathbb{R}^N $ is a bounded function
satisfying
$$ |H(\xi_1)-H(\xi_2)|\le C|\xi_1-\xi_2|,\quad \forall \xi_1,\xi_2\in \mathbb{R}^N,\quad C>0.$$
Consider  $w_1, w_2$ positive functions such that $w_1, w_2\in W^{1,r}_0(\Omega)$,  $1<r<\frac{N}{N-2s+1}$, satisfying
\begin{equation}\label{eq:compar1g} \left\{
\begin{array}{rcll}
(-\Delta)^s w_1&\le& \dfrac{1}{(w_1+1)^\a}H(\n w_1)+g,&\hbox{ in  }\Omega,\\
w_1&=&0  &\hbox  {  on   }\partial\Omega,
\end{array} \right.
\end{equation}
and
\begin{equation}\label{eq:compar2g}
\left\{
\begin{array}{rcll}
(-\Delta)^s w_2&\ge& \dfrac{1}{(w_2+1)^\a}H(\n  w_2)+g,&\hbox{ in  }\Omega,\\
w_2&=&0 &\hbox{  on   } \partial \Omega,
\end{array} \right.
\end{equation}
where $\a>0$ and $g\in L^1(\O)$. Then, $w_2\ge w_1$ in $\Omega$.
\end{Proposition}
\begin{pf}
Define $w=w_1-w_2$, then $w \in W^{1,r}_0(\Omega)$ with $1<r<\frac{N}{N-2s+1}$. We have just to show that $w^+=0$.

Using \eqref{eq:compar1g} and \eqref{eq:compar2g}, it follows that
$$
(-\Delta)^s w\le \frac{1}{(w_1+1)^\a}(H(\n\, w_1)- H(\n\, w_2))+\left(\frac{1}{(w_1+1)^\a}- \frac{1}{(w_2+1)^\a}\right)H(\n\, w_2),
$$
then,
$$
(-\Delta)^s w\le C\frac{1}{(w_1+1)^\a}|\n\, w|+\left(\frac{1}{(w_1+1)^\a}- \frac{1}{(w_2+1)^\a} \right) H(\n\, w_2).
$$
Since the second member in the previous inequality is bounded in $L^1(\Omega)$, then using Kato's inequality, we get
$$
(-\Delta)^s w_{+}\le C|\n w_{+}|, \qquad w_{+}=\rm max\{w,0\}\in W^{1,r}_0(\Omega).
$$
Therefore, by using the maximum principle obtained in Theorem 3.1 of \cite{APN}, we reach that $w_{+}\equiv 0$. Hence we conclude.
\end{pf}

Now, let $\beta>0$ be such that $\mu(\l)<\beta<\bar{\mu}(\l)$, $\beta$ is close to $\mu(\l)$ in such a way that $|x|^{-\beta-2s}\in L^2(\O)$, hence $(-\D)^s|x|^{-\beta}\in L^2(\Omega)$.

Define $v(x)=\frac{A}{|x|^\beta}$, since $\a>2s-1$, then $\dfrac{\mu(\l)(\a+1)+ 2s}{\mu(\l)+1}>2s>p$. Hence we get the existence of $\beta>0$ such that $\mu(\l)<\beta<\bar{\mu}(\l)$, $\beta$ closed to $\mu(\l)$ and $\dfrac{\beta(\a+1)+ 2s}{\beta+1}>2s>p$.

Fix $\beta$ as above, then we get the existence of $A>0$ and $c^*>0$ such that if $f(x)\le \frac{1}{|x|^{\beta+2s}}$ and $c<c^*$,
$$
(-\Delta )^s v\ge \lambda \dfrac{v}{|x|^{2s}}+ \dfrac{|\nabla v|^{p}}{v^\a}+ c f \inn \Omega.
$$
Hence $v$ is a supersolution to problem \eqref{gener}.

{\bf Proof of Theorem \ref{super}.}

We will use a monotony argument.

Define now $u_n$ to be the minimal solution to the approximating problem
\begin{equation}\label{gener1}
\left\{
\begin{array}{rcll}
(-\Delta )^s u_n &=&\lambda \dfrac{u_n}{1+\frac{1}{n}u_n}\dfrac{1}{|x|^{2s}}+ \dfrac{|\nabla u_n|^{p}}{(1+\frac 1n |\nabla u_n|^p)(1+u_n)^\a}+ c f &\inn \Omega,\\
u_n&>&0 & \inn\Omega,\\
u_n&=&0 & \inn(\mathbb{R}^N\setminus\Omega).
\end{array}\right.
\end{equation}
Since $v$ is a supersolution to problem \eqref{gener1}, then using the comparison principle in Proposition \ref{lm:comp} it holds that the sequence $\{u_n\}_n$ is increasing in $n$ and $u_n\le v$ for all $n$. Hence we get the existence of a measurable function $u$ such that $u_n\nearrow u$ strongly in $L^\theta(\O)$ for all $\theta<\frac{N}{\beta}$ and $u\le v$ in $\ren$.

To simplify the notation, we set
$$
g_{1n}(x):=\dfrac{|\nabla u_n|^{p}}{(1+\frac 1n |\nabla u_n|^p)(1+u_n)^\a}, \,\, g_{2n}(x):=\lambda \dfrac{u_n}{1+\frac{1}{n}u_n}\dfrac{1}{|x|^{2s}}+ c f,
$$
and $g_n=g_{n1}+g_{n2}$. Notice that $g_{2n}(x)\le \frac{C}{|x|^{\beta+2s}}$ in $\O$.
It is not difficult to show that $\dyle\io (g_{1n}+g_{2n}) \d^s dx \le C$ for all $n$.

We claim that $||g_n||_{L^1(\O)}\le C$ for all $n$. It is clear that $||g_{2n}||_{L^1(\O)}\le C$ for all $n$.

Using the definition of $u_n$, we have
$$
u_n(x)=\io \mathcal{G}_s(x,y) g_n(y)dy\quad \hbox{    and then     } |\n u_n(x)|\le\io |\n_x \mathcal{G}_s(x,y)|g_n(y)dy.
$$
Hence, for $p<\s<2s$, to be chosen later, it follows that
$$
|\n u_n(x)|^{\s}\le \Big(\io |\n_x \mathcal{G}_s(x,y)|g_n(y)dy\Big)^{\s}\le \Big(\io \dfrac{|\n_x \mathcal{G}_s(x,y)|}{\mathcal{G}_s(x,y)} \mathcal{G}_s(x,y)g_n(y)dy\Big)^{\s}.
$$
Recall that $h(x,y)=\max\{\dfrac{1}{|x-y|},\dfrac{1}{\d(x)}\}$, form \eqref{green2}, it holds that
\begin{eqnarray*}
|\n u_n(x)|^{\s} & \le & \dyle \Big(\io
(h(x,y))^{\a} \mathcal{G}_s(x,y)g_n(y)dy\Big)\Big(\io \mathcal{G}_s(x,y)g_n(y)dy\Big)^{\s-1}\\
&\le & \dyle \Big(\io (h^{\s}(x,y)
\mathcal{G}_s(x,y)g_n(y)dy\Big)u^{\s-1}_n(x).
\end{eqnarray*}
Thus
\begin{eqnarray*}
&\dyle \frac{|\n u_n(x)|^{\s}}{(1+u_n)^{\s-1}}\le\dyle \io \Big(h^{\s}(x,y) \mathcal{G}_s(x,y)g_n(y)dy\Big)\frac{u_n^{\s-1}(x)}{(1+u_n(x))^{\s-1}}\\
&\le \dyle \io \Big(h^{\s}(x,y)
\mathcal{G}_s(x,y) g_{1n}(y)dy\Big)\frac{u_n^{\s-1}(x)}{(1+u_n(x))^{\s-1}}+ \io \Big(h^{\s}(x,y) \mathcal{G}_s(x,y)g_{2n}(y)dy\Big)\frac{u_n^{\s-1}(x)}{(1+u_n(x))^{\s-1}}.
\end{eqnarray*}
By integrating in $x$, we get
\begin{eqnarray*}
\dyle\io \frac{|\n u_n(x)|^{\s}}{(1+u_n)^{\s-1}}dx & \le & \dyle \io g_{1n}(y)\Big(\io h^{\s}(x,y) \mathcal{G}_s(x,y)\frac{u_n^{\s-1}(x)}{(1+u_n(x))^{\s-1}}dx\Big)dy\\ & + & \dyle \l \io g_{2n}(y)\Big(\io
h^{\s}(x,y)\mathcal{G}_s(x,y)\frac{u_n^{\s-1}(x)}{(1+u_n(x))^{\s-1}}dx\big)dy\equiv J_1+J_2.
\end{eqnarray*}
Let us begin by estimating $J_1$.
\begin{eqnarray*}
J_1 & \le & \dyle \io  \dfrac{|\nabla u_n(y)|^{p}}{(1+u_n(y))^\a}
\Big(\io \dfrac{\mathcal{G}_s(x,y)}{|x-y|^{\s}}dx +\io
\dfrac{\mathcal{G}_s(x,y)}{\d^{\s}(x)}dx\Big)dy.
\end{eqnarray*}
It is clear that
$$
\dfrac{\mathcal{G}_s(x,y)}{|x-y|^{\s}}\le \dfrac{C(N,s)}{|x-y|^{N-2s+\s}}.
$$
Thus using the fact that $\s<2s$, it holds that
$$
\io\dfrac{\mathcal{G}_s(x,y)}{|x-y|^{\s}}dx\le C(\O,N,s)<\infty.
$$
Define now $\psi(y):=\dyle\io
\dfrac{\mathcal{G}_s(x,y)}{\d^{\s}(x)}dx$, then
$$
\left\{
\begin{array}{rcll}
(-\Delta )^s \psi &=&\dfrac{1}{\d^\s(x)} &\inn \Omega,\\
\psi &=&0 & \inn(\mathbb{R}^N\setminus\Omega).
\end{array}\right.
$$
Using Theorem 1.2 in \cite{Gia}, we obtain that $\psi\backsimeq \d^{2s-\s}$. Therefore combining the above estimates, we reach that
$$
J_1 \le C(\O,N,s)\dyle \io  \dfrac{|\nabla u_n(y)|^{p}}{(1+u_n(y))^\a} dy.
$$
We deal now with $J_2$. Recall that $g_{2n}(y)\le \dfrac{C}{|y|^{\beta+2s}}$, thus
\begin{eqnarray*}
J_2  \le  \dyle \io \frac{C}{|y|^{\beta+2s}}\Big(\io \dfrac{\mathcal{G}_s(x,y)}{|x-y|^{\s}}dx\Big)dy+ \io \frac{C}{|y|^{\beta+2s}}\psi(y)dy.
\end{eqnarray*}
Hence, as in the computations of $J_1$,
$$
J_2 \le C\io \frac{dy}{|y|^{\beta+2s}}dy=C(N,s,\beta,\O)<\infty.
$$
Thus, we conclude that
$$
\dyle\io \frac{|\n u_n(x)|^{\s}}{(1+u_n)^{\s-1}}dx\le C_1\io \dfrac{|\nabla u_n(x)|^{p}}{(1+u_n(x))^\a}dx +C_3.
$$
Now, using Young inequality, it holds that
$$
\dyle\io \frac{|\n u_n(x)|^{\s}}{(1+u_n)^{\s-1}}dx\le \e \dyle\io \frac{|\n u_n(x)|^{\s}}{(1+u_n)^{\s-1}}dx +C(\e)\dyle\io \frac{1}{(1+u_n)^{\frac{\a\s-p(\s-1)}{\s-p}}}dx.
$$
Recall that $\s\in (p,2s)$, since $0<p-(2s-1)<1$, then $p<\frac{p}{p-(2s-1)}$. Thus we can chose $\s$ such that  $p<\s<\max\{2s, \frac{p}{p-(2s-1)}\}$.

 Hence $\a\s-p(\s-1)\ge 0$ and then $\dyle\io \frac{1}{(1+u_n)^{\frac{\a\s-p(\s-1)}{\s-p}}}dx\le |\O|$.

As a conclusion and choosing $\e$ small enough, we obtain that
\begin{equation}\label{eq:mm}
\dyle\io \frac{|\n u_n|^{\s}}{(1+u_n)^{\s-1}}dx+\dyle\io \frac{|\n u_n|^{p}}{(1+u_n)^{\a}}dx\le C\mbox{ for all  }n.
\end{equation}
Hence $||g_{2n}||_{L^1(\O)}\le C$ for all $n$ and the claim follows.

By the compactness result in Theorem \ref{key}, we get the existence of $u\in W^{1,q}_0(\O)$, for all $q<\frac{N}{N-2s+1}$, such that  up to a subsequence, $u_n\to u$ strongly in $W^{1,\theta}_0(\O)$ for all $\theta<\frac{N}{N-2s+1}$ and $|\n u_n|\to |\n u|$ a.e.  in $\O$. Using Fatou's Lemma we reach that
$$
\dyle\io \frac{|\n u|^{\s}}{(1+u)^{\s-1}}dx+\dyle\io \frac{|\n u|^{p}}{(1+u)^{\a}}dx\le C.
$$
Now, since $p<2s$, then going back to estimate \eqref{eq:mm}, choosing $\s\in (p, 2s)$  and using Vitali lemma, we can prove that
$$
\frac{|\n u_n|^{p}}{(1+u_n)^{\a}}\to\frac{|\n u|^{p}}{(1+u)^{\a}} \mbox{  strongly  in  }L^1(\O).
$$
Thus $u$ is a solution to problem \eqref{gener0} with $u\in W^{1,\theta}_0(\O)$ for all $\theta<\frac{N}{N-2s+1}$. \cqd

\begin{remarks}

Under additional hypothesis on $\a$, we can show the existence of a solution to the problem \eqref{gener0} where the term $\dfrac{1}{(1+u)^\a}$ is replaced by $\dfrac{1}{u^\a}$. More precisely, assume that $2s-1<\a<p+1-\frac{p}{s}$ (this is possible using the fact that $s>\frac 12$ and $1<p<2s$). Now, we consider $u_n$ to be the minimal solution to the problem
\begin{equation}\label{gener1s}
\left\{
\begin{array}{rcll}
(-\Delta )^s u_n &=&\lambda \dfrac{u_n}{1+\frac{1}{n}u_n}\dfrac{1}{|x|^{2s}}+ \dfrac{|\nabla u_n|^{p}}{(1+\frac 1n |\nabla u_n|^p)(\frac 1n+u_n)^\a}+ c f &\inn \Omega,\\
u_n&>&0 & \inn\Omega,\\
u_n&=&0 & \inn(\mathbb{R}^N\setminus\Omega).
\end{array}\right.
\end{equation}
As above, $v$ is a supersolution to \eqref{gener1s} and then the increasing sequence $\{u_n\}_n$ satisfies $u_n\le v$, for all $n$. It is clear that the only point that we have to prove is the fact that
$$
\left\|\dfrac{|\nabla u_n|^{p}}{(1+\frac 1n |\nabla u_n|^p)(\frac 1n+u_n)^\a}\right\|_{L^1(\O)}\le C,\mbox{  for all  }n.
$$
Repeating the same computation as above we arrive to
$$
\dyle\io \frac{|\n u_n(x)|^{\s}}{u_n^{\s-1}}dx\le C_1\io \dfrac{|\nabla u_n(x)|^{p}}{(\frac 1n+u_n(x))^\a}dx +C_3.
$$
Then by Young inequality,
$$
(1-\e)\dyle\io \frac{|\n u_n(x)|^{\s}}{u_n^{\s-1}}dx\le C(\e)\dyle\io \frac{1}{u_n^{\frac{\a\s-p(\s-1)}{\s-p}}}dx +C_3.
$$
Since $2s-1<\a<p+1-\frac{p}{s}$, then $p<\frac{p}{p-\a}$, hence choosing  $\s$ such that
$$
\max\{p, \frac{p}{p+1-\a}\}<\s<\frac{p}{p-\a}<2s,
$$
it holds that $\a\s-p(\s-1)>0$. Now, using the fact that the sequence $\{u_n\}_n$ is increasing in $n$ and since $f\gneqq 0$, then $u_n\ge u_1\ge C\d^s$ for some universal constant and then
$$
\io \frac{1}{u_n^{\frac{\a\s-p(\s-1)}{\s-p}}}dx\le C\io \dfrac{1}{\d^{s\frac{\a\s-p(\s-1)}{\s-p}}}dx.
$$
Since $s\frac{\a\s-p(\s-1)}{\s-p}<1$, then $\io \frac{1}{\d^{s\frac{\a\s-p(\s-1)}{\s-p}}}dx<\infty$ and hence we conclude.

In a forthcoming work, we will analyze the general case without using monotony arguments and under general integrability assumptions on $f$.

\end{remarks}

\end{document}